\documentclass{amsart}

\usepackage{amsfonts,amssymb,amsmath,amsthm}
\usepackage{array,booktabs}
\usepackage{tikz}
\usepackage{tikz-cd}
\usepackage{verbatim}
\usepackage{caption,subcaption}
\usepackage{longtable}

\newcommand{\kk}{\mathbb{K}}

\newcommand{\m}{\mathbf{m}}

\newcommand{\D}{\Delta}

\newcommand{\cJ}{\mathcal{J}}
\newcommand{\cR}{\mathcal{R}}

\newcommand{\Z}{\mathbb{Z}}
\newcommand{\A}{\mathcal{A}}
\DeclareMathOperator{\syz}{syz}

\newtheorem{thm}{Theorem}[section]
\newtheorem{cor}[thm]{Corollary}
\newtheorem{lem}[thm]{Lemma}
\newtheorem{prop}[thm]{Proposition}

\theoremstyle{definition}
\newtheorem{defn}[thm]{Definition}
\newtheorem{exm}[thm]{Example}
\newtheorem{notation}[thm]{Notation}
\newtheorem{assumptions}[thm]{Assumptions}

\theoremstyle{remark}
\newtheorem{remark}[thm]{Remark}

\usepackage{hyperref}

\graphicspath{{Figures/}}

\title{Free and non-free multiplicities on the $A_3$ arrangement}

\author[M.~DiPasquale]{Michael DiPasquale}     
\address{Michael DiPasquale\\     
	Department of Mathematics\\     
	Oklahoma State University\\     
	Stillwater\\
	OK \ 74078-1058\\     
	USA}     
\email{mdipasq@okstate.edu}
\urladdr{\url{http://math.okstate.edu/people/mdipasq/}}   
\author[C.A.~Francisco]{Christopher A. Francisco}
\address{Christopher A. Francisco\\
	Department of Mathematics\\     
	Oklahoma State University\\     
	Stillwater\\
	OK \ 74078-1058\\     
	USA}    
\email{chris.francisco@okstate.edu}
\urladdr{\url{https://math.okstate.edu/people/chris/}}
\author[J.~Mermin]{Jeffrey Mermin}
\address{Jeffrey Mermin\\     
	Department of Mathematics\\     
	Oklahoma State University\\     
	Stillwater\\
	OK \ 74078-1058\\     
	USA}
\email{mermin@math.okstate.edu}     
\urladdr{\url{https://math.okstate.edu/people/mermin/}}   
\author[J.~Schweig]{Jay Schweig}
\address{Jay Schweig\\
	Department of Mathematics\\     
	Oklahoma State University\\     
	Stillwater\\
	OK \ 74078-1058\\     
	USA}
\email{jay.schweig@okstate.edu}
\urladdr{\url{https://math.okstate.edu/people/jayjs/}}
\thanks{}
\thanks{}

\begin{document}
	
\begin{abstract}
We give a complete classification of free and non-free multiplicities on the $A_3$ braid arrangement.  Namely, we show that all free multiplicities on $A_3$ fall into two families that have been identified by Abe-Terao-Wakefield (2007) and Abe-Nuida-Numata (2009).  The main tool is a new homological obstruction to freeness derived via a connection to multivariate spline theory.
\end{abstract}

\maketitle

\section{Introduction}

Let $V=\kk^{\ell}$ be a vector space over a field $\kk$ of characteristic zero.  A \textit{central hyperplane arrangement} $\A=\{H_1,\ldots,H_n\}$ is a set of hyperplanes $H_i\subset V$ passing through the origin in $V$.  In other words, if we let $\{x_1,\ldots,x_{\ell}\}$ be a basis for the dual space $V^*$ and $S=\mbox{Sym}(V^*)\cong \kk[x_1,\ldots,x_l]$, then $H_i=V(\alpha_{H_i})$ for some choice of linear form $\alpha_{H_i}\in V^*$, unique up to scaling.  A \textit{multi-arrangement} is a pair $(\A,\m)$ of a central arrangement $\A$ and a map $\m:\A\rightarrow Z_{\ge 0}$, called a multiplicity.  If $\m\equiv 1$, then $(\A,\m)$ is denoted $\A$ and is called a \textit{simple} arrangement.

The \textit{module of derivations} on $S$ is defined by $\mbox{Der}_{\kk}(S)=\bigoplus_{i=1}^{\ell} S\partial_{x_i}$, the free $S$-module with basis $\partial_{x_i}=\partial/\partial x_i$ for $i=1,\ldots,\ell$.  The module $\mbox{Der}_\kk(S)$ acts on $S$ by partial differentiation.  Our main object of study is the module $D(\A,\m)$ of \textit{logarithmic derivations} of $(\A,\m)$:
\[
D(\A,\m):=\{\theta\in \mbox{Der}_\kk(S): \theta(\alpha_H)\in\langle \alpha_H^{\m(H)}\rangle \text{ for all }H\in\A \},
\]
where $\langle \alpha_H^{\m(H)} \rangle\subset S$ is the ideal generated by $\alpha_H^{\m(H)}$.  If $D(\A,\m)$ is a free $S$-module, then we say $(\A,\m)$ is free or $\m$ is a free multiplicity of the simple arrangement $\A$.  For a simple arrangement, $D(\A,\m)$ is denoted $D(\A)$; if $D(\A)$ is free we say $\A$ is free.

The module of logarithmic derivations is central to the theory of hyperplane arrangements, initiated and studied by Saito in~\cite{SaitoUniformization,SaitoLogDiff}.  In particular, it is important to know when $\A$ is a free arrangement. Indeed, possibly the most important open question in hyperplane arrangements is whether freeness is a combinatorial property; see, for instance, \cite{OrlikTerao}. Yoshinaga~\cite{YoshCharacterizationFreeArr} has shown that freeness of an arrangement is closely related to freeness of the canonical restricted multi-arrangement defined by Ziegler~\cite{ZieglerMulti}.  Hence the freeness of multiarrangements is important to the theory of hyperplane arrangements as well. 

The \textit{braid arrangement of type} $\A_{\ell}$ is defined as $\{H_{ij}=V(x_i-x_j): 0\le i<j\le \ell\}$ in $V\cong \kk^{\ell+1}$.  Free multiplicities on braid arrangements have been studied in~\cite{TeraoMultiDer,TeraoDoubleCoxeter,AbeQuasiConstant,YoshinagaPrimitiveDerivationMultiCoxeter,AbeSignedEliminable}.  Until recently there have been very few tools to study multi-arrangements.  In two papers~\cite{TeraoCharPoly,EulerMult}, Abe-Terao-Wakefield extend the theory of the characteristic polynomial and deletion-restriction arguments to multi-arrangements.  These allow new methods for determining the freeness (and non-freeness) of multiarrangements.  In particular, the tool of local and global mixed products is introduced for characterizing non-freeness of multi-arrangements in some instances.  Abe~\cite{AbeDeletedA3} uses these tools to give the first non-trivial complete classification of free and non-free multiplicities on a hyperplane arrangement, the so-called deleted $A_3$ arrangement.  The main result of this paper is the next natural step; namely a complete characterization of free and non-free multiplicities on the $A_3$ braid arrangement.

% %Intro changes mike
There are two main classes of multiplicities that have been characterized as free on the $A_3$ braid arrangement.  The first class may be described as follows.  Suppose that, for some index $i$, the inequalities $\m(H_{jk})\ge \m(H_{ij})+\m(H_{ik})-1$ are satisfied for every pair of distinct indices $j\neq i,k\neq i$ (geometrically, three hyperplanes which intersect in codimension two have relatively high multiplicity compared to the other three hyperplanes).  If these inequalities are satisfied, we say that the index $i$ is a \textit{free vertex} for $\m$.  If $\m$ has a free vertex, then it is known that $\m$ is a free multiplicity~\cite[Corollary~5.12]{EulerMult} (see also Corollary~\ref{cor:FirstFreeCondition}).  To describe the second (much more complex) class of free multiplicities, take four non-negative integers $n_0,n_1,n_2,$ and $n_3$ and consider the multiplicity $\m(H_{ij})=n_i+n_j+\epsilon_{ij}$, where $\epsilon_{ij}\in\{-1,0,1\}$.  We call these ANN multiplicities, due to a classification of all such multiplicities as free or non-free by Abe, Nuida, and Numata in~\cite{AbeSignedEliminable}.  It turns out the multiplicity $\m(H_{ij})=n_i+n_j$ is always free, and the classification of all ANN multiplicities depends on measuring the deviation from these using \textit{signed-eliminable graphs}.  We describe this classification in more detail in Section~\ref{sec:ANN}.  Our main result is that all free multiplicities on $A_3$ fall into these two classes.

\begin{thm}\label{thm:intro}
The multi-braid arrangement $(A_3,\m)$ is free if and only if $\m$ has a free vertex or $\m$ is a free ANN multiplicity.
\end{thm}

We prove Theorem~\ref{thm:intro} via a connection to multivariate splines first noted by Schenck in~\cite{HalSplit} and further developed by the first author in~\cite{GSplinesGraphicArrangements}.  Our main tool, Theorem~\ref{thm:FreeEquiv}, is a new criterion for freeness of a multi-braid arrangement $(A_3,\m)$ in terms of syzygies of ideals generated by powers of the linear forms defining the hyperplanes of $A_3$.  This condition gives a robust obstruction to freeness which we use to establish Theorem~\ref{thm:intro}.

Our paper is arranged as follows.  In Section~\ref{sec:notation}, we introduce the notation and background we will use throughout the paper.  Section~\ref{sec:technical} uses homological techniques to prove Theorem~\ref{thm:FreeEquiv}, which says that the multi-arrangement $(A_3,\m)$ is free precisely when a certain syzygy module is ``locally generated.''
Readers may safely skip the rest of that section and simply read the theorem statement if they desire.  In Sections~\ref{sec:PartI} and~\ref{sec:PartII}, we prove Theorem~\ref{thm:intro} using Theorem~\ref{thm:FreeEquiv} along with combinatorial arguments using syzygies and Hilbert functions. In Section~\ref{sec:ANN}, we recover the non-free multiplicities in the classification of Abe-Nuida-Numata \cite{AbeSignedEliminable}.  We conclude with remarks on using the free ANN multiplicities of~\cite{AbeSignedEliminable} to construct minimal free resolutions for certain ideals generated by powers of linear forms.  In Appendix~\ref{app}, we illustrate the classification of Theorem~\ref{thm:intro} in the case of two-valued multiplicities.

\section{Notation and preliminaries} \label{sec:notation}

%In this section we introduce graphic arrangements and a chain complex which we will use to study logarithmic derivations on graphic arrangements.
In this section we set up the main notation to be used throughout the paper.  The data of the $A_3$ arrangement is captured in a labeling of the vertices of $K_4$, the complete graph on four vertices; namely the edge between $v_i$ and $v_j$ in $K_4$ corresponds to the hyperplane $H_{ij}=V(x_i-x_j)$.  As such we will also denote $A_3$ by $\A_{K_4}$.  Put $m_{ij}=\m(H_{ij})$.  We will record the multiplicities of the hyperplanes as a lexicographically ordered list $\m=(m_{01},m_{02},m_{03},m_{12},m_{13},m_{23})$ which we can also associate to the obvious labelling of the edges of $K_4$.  We will often refer to the multiplicities as $a,b,c,d,e,f$ according to the edge-labeling in Figure~\ref{fig:k4-early}.

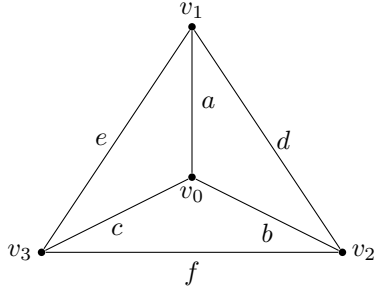
\begin{figure}[htp]
	\centering
	\begin{tikzpicture}[scale=1]
	\tikzstyle{dot}=[circle,fill=black,inner sep=1 pt];
	
	\node[dot] (0) at (0,0) {};
	\node[dot] (1) at (0,2){};
	\node[dot] (2) at (2,-1){};
	\node[dot] (3) at (-2,-1){};
	
	\draw (0)node[below]{$v_0$}--node[right]{$a$}(1) node[above]{$v_1$}-- node[left]{$e$}(3) node[left]{$v_3$};
	\draw (0) --node[below]{$c$} (3) -- node[below]{$f$}(2)node[right]{$v_2$}-- node[below]{$b$} (0);
	\draw (1)--node[right]{$d$} (2);
	\end{tikzpicture}
	\caption{Labelling Convention}\label{fig:k4-early}
\end{figure}

For simplicity, we set $S=\kk[x_0,x_1,x_2,x_3]$ and $\alpha_{ij} = x_i-x_j$ for all $i > j$.  Our goal is to study when the module of multi-derivations
\[
D(A_3,\m)=\{\theta\in\mbox{der}_\kk(S): \theta(\alpha_{ij})\in\langle\alpha_{ij}^{m_{ij}}\rangle \mbox{ for all } 0\le i<j\le 3 \}
\]
is free as an $S$-module.

\begin{remark}\label{rem:essA3}
	Note that there is a line contained in every hyperplane of $A_3$, namely the line described parametrically as $\{(t,t,t,t): t\in \kk\}$.  Thus $A_3$ is not \textit{essential}; an essential arrangement is one in which all hyperplanes intersect in only the origin.  Projecting along this line we obtain an arrangement in $\kk^3$ whose hyperplanes may be described as follows.  Set $x=x_1-x_0,y=x_2-x_0,z=x_3-x_0$.  Then the \textit{essential} $A_3$ arrangement in $\kk^3$ is $A_3^e=\{V(x),V(y),V(z),V(y-x),V(z-x),V(z-y)\}$.  See Figure~\ref{fig:A3} for a picture of this arrangement in $\mathbb{R}^3$.  Set $R=\kk[x,y,z]$.  It is not difficult to see that $D(A_3,\m)\cong D(A^e_3,\m)\otimes_R S$.  Hence freeness of $(A_3,\m)$ and $(A^e_3,\m)$ are equivalent.  We will suppress the distinction between $A^e_3$ and $A_3$, calling both the $A_3$ arrangement.  We will also suppress the distinction between the polynomial rings $S=\kk[x_0,x_1,x_2,x_3]$ and $R=\kk[x,y,z]$, simply letting $S$ refer to the ambient polynomial ring in both situations.  It will be obvious from context (but not important) which polynomial ring is meant.
\end{remark}

\begin{figure}[htp]
\includegraphics[scale=.25]{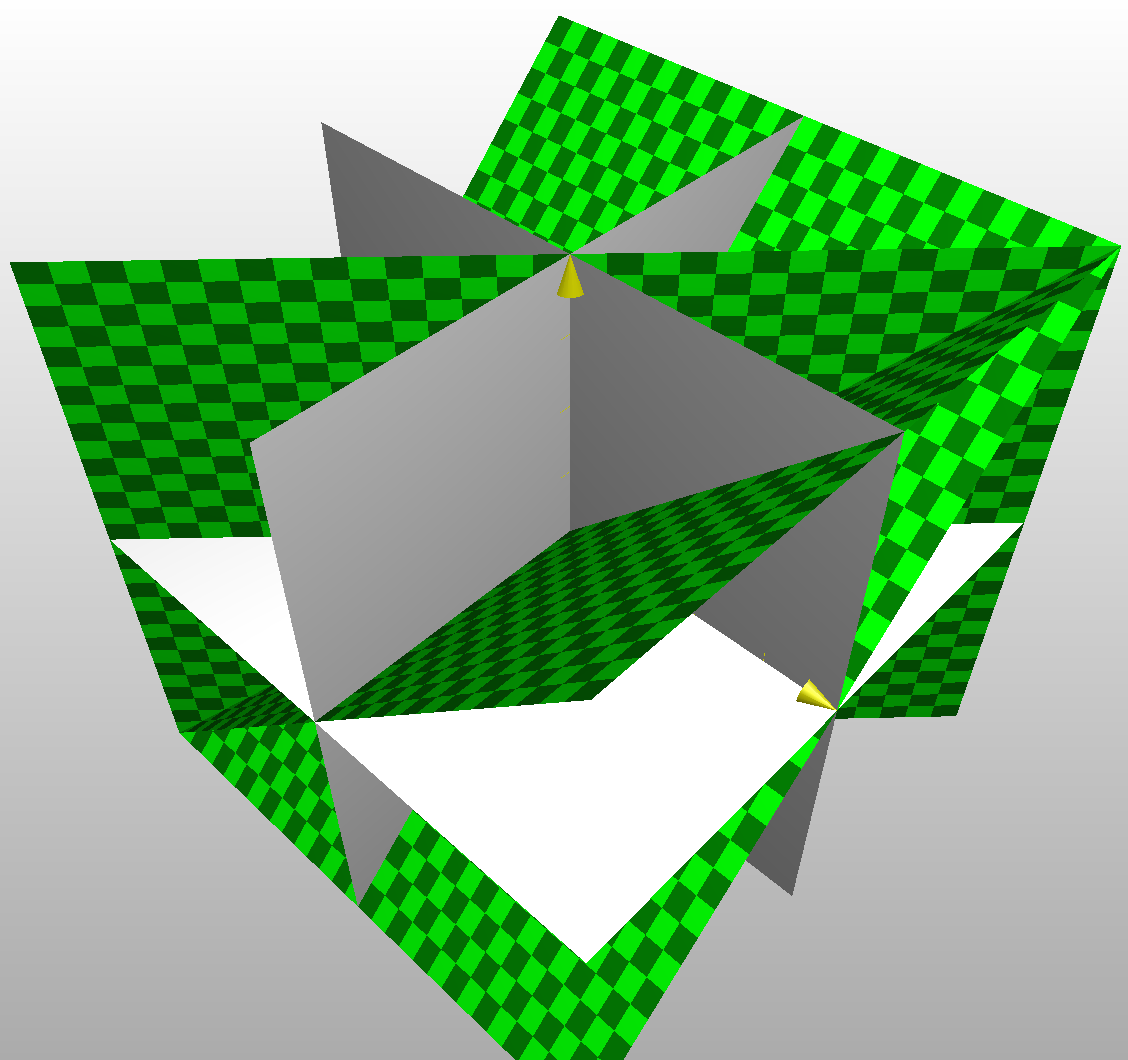}
\caption{Essential $A_3$ arrangement}\label{fig:A3}
\end{figure}

In the next section, which is the technical heart of the paper, we will show that freeness of $D(A_3,\m)$ is determined by syzygies of certain ideals which we now define.  For any edge $e=\{i,j\}$, we set the ideal $J(ij) = \langle \alpha_{ij}^{m_{ij}}\rangle$.  More generally, for any subset $\sigma \subset \{0,1,2,3\}$, we set \[J(\sigma)=\sum_{\{i,j\} \subset \sigma} J(ij).\] For instance, $J(012) = J(01)+J(02)+J(12)$.  Using $x,y,z$ in place of $x_1-x_0,x_2-x_0,x_3-x_0$ as in Remark~\ref{rem:essA3} and the multiplicity labels $(a,b,c,d,e,f)$ as in Figure~\ref{fig:k4-early}, the following is a list of all ideals $J(\sigma)$ for $\sigma\subset\{0,1,2,3\}$.
\[
\begin{array}{ll}
J(01)=\langle x^a \rangle &  J(012)=\langle x^a,y^b,(x-y)^d \rangle \\
J(02)=\langle y^b \rangle & J(013)=\langle x^a,z^c,(x-z)^e \rangle \\
J(03)=\langle z^c \rangle  &  J(023)=\langle y^b,z^c,(y-z)^f \rangle \\
J(12)=\langle (x-y)^d \rangle & J(123)=\langle (x-y)^d,(x-z)^e,(y-z)^f \rangle \\
J(13)=\langle (x-z)^e \rangle & J(0123)=\langle x^a,y^b,z^c,(x-y)^d,(x-z)^e,(y-z)^f \rangle \\
J(23)=\langle (y-z)^f \rangle & \\
\end{array}
\]
Theorem~\ref{thm:FreeEquiv} will show that the freeness of the multi-arrangement $(A_3,\m)$ depends on the relationship between the ``global'' first syzygy module $\syz(J(0123))$ and its ``local'' first syzygies $\syz(J(ijk))$, for $0\le i<j<k\le 3$.

%\begin{figure}[htp]
%	\centering
%	\begin{tikzpicture}[scale=1.4]
%	\tikzstyle{dot}=[circle,fill=black,inner sep=1 pt];
%	
%	\node[dot] (0) at (0,0) {};
%	\node[dot] (1) at (0,2){};
%	\node[dot] (2) at (2,-1){};
%	\node[dot] (3) at (-2,-1){};
%	%\node at (-2/3,1/3){$\sigma_1$};
%	%\node at (0,-2/3){$\sigma_2$};
%	
%	\draw (0)node[below]{$v_0$}--node{$(x_1-x_0)^a$}(1) node[above]{$v_1$}-- node[left]{$(x_3-x_1)^e$}(3) node[left]{$v_3$};
%	\draw (0) --node{$(x_3-x_0)^c$} (3) -- node[below]{$(x_3-x_2)^f$}(2)node[right]{$v_2$}-- node{$(x_2-x_0)^b$} (0);
%	\draw (1)--node[right]{$(x_2-x_1)^d$} (2);
%	\end{tikzpicture}
%	\caption{Labelling Convention}\label{fig:k4-early}
%\end{figure}

\section{Technical machinery} \label{sec:technical}

The bulk of this section is technical, and the goal is simply to prove Theorem~\ref{thm:FreeEquiv}. Later sections require only the statement of this theorem, so readers wishing to avoid the technical details can safely skip to Section~\ref{sec:hf}. In particular, additional notation introduced in this section is not used elsewhere in the paper.

A \textit{graphic arrangement} is a subarrangement of a braid arrangement.  More precisely, let $G$ be a vertex-labeled graph on $\ell+1$ vertices $\{v_0,\ldots,v_{\ell}\}$ with no loops or multiple edges.  Denote by $E(G)$ the set of edges of $G$.  We denote the edge between vertices $v_i,v_j$ by $\{i,j\}$.  The graphic arrangement corresponding to $G$ is
\[
\A_G=\bigcup_{\{i,j\}\in E(G)} V(x_j-x_i)\subset\kk^{\ell+1}.
\]
A graphic multi-arrangement $(\A_G,\m)$ is a graphic arrangement $\A_G$ with an assignment $\m:E(G)\to \mathbb{N}$ of a positive integer $\m(e)$ to every edge $e\in E(G)$.

\subsection{Homological necessities} Our main tool to study freeness of $D(\A_G,\m)$ is a chain complex $\cR/\cJ[G]$ whose top homology is the module $D(\A_G,\m)$, introduced in~\cite{GSplinesGraphicArrangements}.  We now define this complex.

Denote by $\D(G)$ the clique complex of $G$.  This is the simplicial complex on the vertex set of $G$ whose simplices are given by sets of vertices that induce a complete subgraph (clique) of $G$.  Denote by $\D(G)_i$ the set of cliques of $G$ with $(i+1)$ vertices, i.e., the simplices of $\D(G)$ of dimension $i$.  
%Assume a labeling $\{v_0,\ldots, v_n\}$ of the vertices of $G$ is fixed.  We will identify a simplex $\sigma\in\D(G)_i$ with the indices of its vertices $\{j_0,\ldots,j_i\}$ listed in increasing order.

\begin{defn}\label{defn:TopComplex1}
Let $G$ be a graph with $\ell+1$ vertices and set $S=\kk[x_0,\ldots,x_{\ell}]$.  Define the complex $\cR[G]$ to be the simplicial co-chain complex of $\D(G)$ with coefficients in $S$; that is, $\cR[G]_i=\bigoplus\limits_{\gamma\in\D(G)_i} S[e_{\gamma}]$, where $[e_{\gamma}]$ is a formal symbol corresponding to the $i$-dimensional clique $\gamma$. The differential $\delta^i: \cR[G]_i\rightarrow \cR[G]_{i+1}$ is the simplicial differential of the co-chain complex of $\Delta(G)$ with coefficients in $S$. 
%For concreteness, we take a basis $[e_{\gamma}]$.
\end{defn}

\begin{remark}
By definition, $H^\bullet(\cR[G])$ is isomorphic to the cohomology of $\Delta(G)$ with coefficients in $S$.
\end{remark}

In the following definition, if $e=\{i,j\}\in E(G)$, we will denote $\alpha_{ij}=x_i-x_j$ by $\alpha_e$.

\begin{defn}\label{defn:Jideal}
Let $(\A_G,\m)$ be a graphic multi-arrangement.  Let $\sigma=\{j_0,j_1,\ldots,j_i\}$ be a clique of $G$.  Then
\[
J(\sigma):=\langle \alpha_e^{m_e} |e\in E(\sigma) \rangle.
\]
If $\sigma$ is a vertex of $G$, then $J(\sigma)=0$.
\end{defn}

\begin{defn}\label{defn:derivationComplex}
Given a graphic multi-arrangement $(\A_G,\m)$, $\cJ[G]$ is the sub-chain complex of $\cR[G]$ with $\cJ[G]_i=\sum\limits_{\gamma\in\D(G)_i} J(\gamma)[e_{\gamma}]$.  $\cR/\cJ[G]$ denotes the quotient complex $\cR[G]/\cJ[G]$ with
$\cR/\cJ[G]_i=\bigoplus\limits_{\gamma\in\D(G)_i} (S/J(\gamma))[e_{\gamma}]$.
\end{defn}

\begin{lem}\label{lem:firstCohomology}
The module of multi-derivations $D(\A_G,\m)$ of the graphic multi-arrangement $(\A_G,\m)$ is $H^0(\cR/\cJ[G])$.
\end{lem}
\begin{proof}
Let $F \in \cR[G]_0$. Write $F=(\ldots, F_v, \ldots)_{v \in V(G)}$. Then $F \in \ker(\bar{\delta}^0)$ if and only if, for all $e=\{i,j\}\in E(G)$, we have \[(\delta(F))_e = F_i - F_j \in J(e)=\langle(x_i-x_j)^{m(e)}\rangle.\] This last statement is the definition of $D(\A_G,\m)$.
%Let $F\in \bigoplus_{v\in V(G)} S$.  Then $F\in \mbox{ker}(\overline{\delta}^0)\iff F_i-F_j=0\in S/J(e) \iff F_i-F_j\in \langle \alpha_e^{\m(e)}\rangle$ for every edge $e=\{i,j\}\in E(G)$.  
\end{proof}

With Lemma~\ref{lem:firstCohomology} as our justification, we will call $\cR/\cJ[G]$ the \textit{derivation complex} of $G$.

\begin{exm}\label{ex:ThreeCycleHomologies}
Take $G$ to be the three-cycle with labeling as in Figure~\ref{fig:3cycle}.  

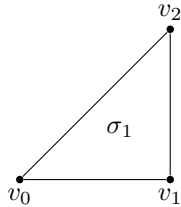
\begin{figure}[htp]
	\centering		
	\begin{tikzpicture}[scale=1]
	\tikzstyle{dot}=[circle,fill=black,inner sep=1 pt];
	
	\node[dot] (1) at (0,0) {};
	\node[dot] (2) at (2,0){};
	\node[dot] (3) at (2,2){};
	\node at (4/3,2/3){$\sigma_1$};
	\draw (1)node[below]{$v_0$}--(2) node[below]{$v_1$}-- (3) node[above]{$v_2$}--(1);
	
	\end{tikzpicture}
	\caption{Three-cycle for Example~\ref{ex:ThreeCycleHomologies}}\label{fig:3cycle}	
\end{figure}

The short exact sequence of complexes $0\rightarrow \cJ[G] \rightarrow \cR[G] \rightarrow \cR/\cJ[G] \rightarrow 0$ is shown below. 
\begin{center}
\begin{tikzcd}
\cJ[G] & & 0 \ar{r}{\delta^0}\ar{d} & J(01)\oplus J(02)\oplus J(12) \ar{d}\ar{r}{\delta^1} & J(012) \ar{d}\ar{r} & 0   \\
\cR[G] &  0 \ar{r} & S^3 \ar{d}\ar{r}{\delta^0} & S\oplus S\oplus S \ar{r}{\delta^1}\ar{d} & S \ar{r}\ar{d} & 0  \\
\cR/\cJ[G] &  0 \ar{r} & S^3 \ar{r}{\overline{\delta^0}} & \dfrac{S}{J(01)}\oplus \dfrac{S}{J(02)} \oplus \dfrac{S}{J(12)} \ar{r}{\overline{\delta^1}} & \dfrac{S}{J(012)} \ar{r} & 0
\end{tikzcd}
\end{center}

The differentials are
\[
\delta^0=\bordermatrix{
  & 0 & 1 & 2 \cr
01 & -1 & 1 & 0  \cr
02 & -1  & 0  & 1 \cr
12 & 0  & -1  & 1  \cr
}
\qquad
\delta^1= \,\bordermatrix{
& 01 & 02 & 12 \cr
012 & 1 & -1 & 1 \cr
%	& 012  \cr
%	01 &   1  \cr
%	02 &  -1  \cr
%	12 & 1    \cr
}.
\]
The homologies $H^i(\cR[G])$ vanish for $i=1,2$, and $H^0(\cR[G])=S$.  The corresponding long exact sequence in (co)homology splits up to yield the short exact sequence
\[
0\rightarrow S \rightarrow H^0(\cR/\cJ[G]) \rightarrow H^1(\cJ[G])\rightarrow 0,
\]
and an isomorphism $H^1(\cR/\cJ[G])\cong H^2(\cJ[G])=0$.  The short exact sequence actually splits, so $H^0(\cR/\cJ[G])\cong S\oplus H^1(\cJ[G])$.

The map $\delta^1:J(01)\oplus J(02) \oplus J(12)\rightarrow J(012)$ is surjective by definition, hence $H^2(\cR/\cJ[G])=0$.  Also, $H^1(\cJ[G])=\mbox{ker}(\delta^1)=\mbox{syz}(J(012))$, the module of syzygies on $J(012)$.  Hence $H^0(\cR/\cJ[G])=D(\A_G,\m)\cong S\oplus \mbox{syz}(J(012))$.
\end{exm}

\begin{remark}
The ideal $J(012)$ in Example~\ref{ex:ThreeCycleHomologies} is codimension two and Cohen-Macaulay.  Hence $D(\A_G,\m)\cong S\oplus \mbox{syz}(J(012))$ is a free module regardless of the choice of $m_{01},m_{02},m_{12}$.  It is well-known that rank two arrangements are totally free for the same reason; they are second syzygy (or \textit{reflexive}) modules of rank two.
\end{remark}

\begin{remark}\label{rem:syzygies}
In Example~\ref{ex:ThreeCycleHomologies}, we understand $\syz(J(012))$ to represent syzygies among the generators $(x_0-x_1)^{m_{01}},(x_1-x_2)^{m_{12}},(x_0-x_2)^{m_{02}}$, even if this is not a minimal generating set.  For instance, if $m_{01}+m_{12}\le m_{02}+1$, then $J(012)$ is generated by $(x_0-x_1)^{m_{01}}$ and $(x_1-x_2)^{m_{12}}$.  In this case, $\syz(J(012))$ is generated by the Koszul syzygy on $(x_0-x_1)^{m_{01}},(x_1-x_2)^{m_{12}}$ and the relation of degree $m_{02}$ expressing $(x_0-x_2)^{m_{02}}$ as a polynomial combination of $(x_0-x_1)^{m_{01}},(x_1-x_2)^{m_{12}}$.
\end{remark}

In Example~\ref{ex:ThreeCycleHomologies}, $H^i(\cR/\cJ[G])=0$ for $i=1,2$, and $D(\A_G,\m)$ was free.  This is no coincidence.

\begin{thm}\cite[Theorem~3.2]{GSplinesGraphicArrangements}\label{thm:free1}
The graphic multi-arrangement $(\A_G,\m)$ is free if and only if $H^i(\cR/\cJ[G])=0$ for all $i>0$.
\end{thm}

Theorem~\ref{thm:free1} follows from a result of Schenck using a Cartan-Eilenberg spectral sequence~\cite{Spect}.  Although we use Theorem~\ref{thm:free1} in this paper primarily to study the multi-braid arrangements $(A_3,\m)$, we show in the following example how it may be used to classify free multiplicities on other graphic arrangements.

\begin{exm}[Deleted $A_3$ arrangement]\label{exm:FreeDepends}
Consider the graph $G$ in Figure~\ref{fig:sqdiag}.  This is the simplest example of a graph where freeness of $(\A_G,\m)$ depends on the multiplicities $\m$.

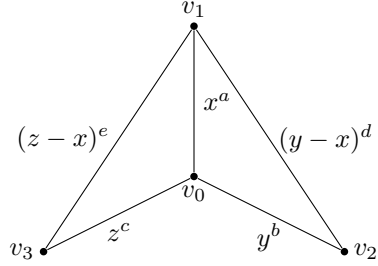
\begin{figure}[htp]
	\centering
	\begin{tikzpicture}[scale=1]
	\tikzstyle{dot}=[circle,fill=black,inner sep=1 pt];
	
	\node[dot] (0) at (0,0) {};
	\node[dot] (1) at (0,2){};
	\node[dot] (2) at (2,-1){};
	\node[dot] (3) at (-2,-1){};
	%\node at (-2/3,1/3){$\sigma_1$};
	%\node at (0,-2/3){$\sigma_2$};
	
	\draw (0)node[below]{$v_0$}--node[right]{$x^a$}(1) node[above]{$v_1$}-- node[left]{$(z-x)^e$}(3) node[left]{$v_3$};
	\draw (3) --node[below]{$z^c$} (0) -- node[below]{$y^b$}(2)node[right]{$v_2$}-- node[right]{$(y-x)^d$} (1);
	\end{tikzpicture}
	\caption{Graph for the deleted $A_3$ arrangement}\label{fig:sqdiag}
\end{figure}

The maps in cohomology (the differentials of $\cR[G]$) are given by
\[
\delta^0=
\bordermatrix{
   & 0  & 1 & 2 & 3 \cr
01 & -1 & 1 & 0 & 0\cr
02 & -1 & 0 & 1 & 0\cr
03 & -1 & 0 & 0 & 1\cr
12 & 0  & -1& 1 & 0\cr
13 & 0  & -1& 0 & 1\cr
}
\qquad
\delta^1=
\bordermatrix{
    & 01 & 02 & 03 & 12 & 13 \cr
013 & 1  & 0  & -1 & 0  & 1 \cr
012 & 1  & -1 &  0 & 1  & 0 \cr
},
\]
where the rows and columns are labeled by faces (see Figure~\ref{fig:sqdiag}).  Let us set $x=x_1-x_0=\alpha_{01},y=x_2-x_0=\alpha_{02},z=x_3-x_0=\alpha_{03}$.  Then $\alpha_{23}=x_3-x_2=z-y$ and $\alpha_{13}=x_3-x_1=z-x$.  Suppose the edge $\{i,j\}$ is assigned multiplicity $m_{ij}$.  Set $m_{01}=a,m_{02}=b,m_{03}=c, m_{12}=d, m_{13}=e$.  We have 
%$J(ij)=\langle \alpha_{ij}^{m_{ij}} \rangle$ (see the edge labels in Figure~\ref{fig:sqdiag}) and
\[
\begin{array}{rl}
J(012) & =\langle x^a,y^b,(y-x)^d \rangle\\
J(013) & =\langle x^a,z^c,(z-x)^e \rangle.
\end{array}
\]
\end{exm}

\begin{remark}
The following characterization of free multiplicities on the deleted $A_3$ arrangement in Example~\ref{exm:FreeDepends} is derived in~\cite{AbeDeletedA3} using techniques for multi-arrangements developed in~\cite{TeraoCharPoly,EulerMult}.  We show how this characterization may be obtained homologically from Theorem~\ref{thm:free1}.
\end{remark}

\begin{prop}
Let $\A_G$ be the deleted $A_3$ arrangement from Example~\ref{exm:FreeDepends}.  With notation as in Example~\ref{exm:FreeDepends}, $(\A_G,\m)$ is free if and only if either $c+e\le a+1$ or $b+d\le a+1$.
\end{prop}
\begin{proof}
We have $H^i(\cR[G])=0$ for $i>0$ since $\D(G)$ is contractible, and $H^1(\cR/\cJ[G])\cong H^2(\cJ[G])$ via the long exact sequence corresponding to $0\rightarrow \cJ[G] \rightarrow \cR[G] \rightarrow \cR/\cJ[G] \rightarrow 0$.  The complex $\cJ[G]$ has the form
\[
\bigoplus_{\{i,j\}\in G} J(ij) \xrightarrow{\delta^1} J(013) \oplus J(023).
\]
The map $\delta^1$ is given by the matrix
\[
\delta^1=
\bordermatrix{
	& 01 & 02 & 03 & 12 & 13 \cr
	013 & 1  & 0  & -1 & 0  & 1 \cr
	012 & 1  & -1 &  0 & 1  & 0 \cr
}.
\]
Let us determine when $\delta^1$ is surjective, hence when $H^2(\cJ[G])\cong H^1(\cR/\cJ[G])=0$.  We see that, given $(f_1,f_2,f_3,f_4,f_5)\in \cJ[G]_1$,  $\delta^1(f_1,f_2,f_3,f_4,f_5)=(f_1+f_3-f_5,f_1-f_2-f_4)$.  This map surjects onto $J(\sigma_1)\oplus J(\sigma_2)$ if and only if either $J(013)$ is generated by $x^{a},z^{c}$ or $J(012)$ is generated by $y^{b}, (y-x)^{d}$.  This in turn happens if and only if either $c+e\le a+1$ or $b+d\le a+1$.  By Theorem~\ref{thm:free1}, $(\A_G,\m)$ is free if and only if $c+e\le a+1$ or $b+d\le a+1$.
\end{proof}

\begin{remark}
Let $G$ be a graph on $n+1$ vertices.  As a consequence of Theorem~\ref{thm:free1}, the Hilbert polynomial (and indeed Hilbert function) of $D(\A_G,\m)$, when $(\A_G,\m)$ is free, is given by the Euler characteristic of $\cR/\cJ[G]$, namely
\[
\begin{array}{rl}
HP(D(\A_G,\m),d) & =\sum\limits_{i=0}^{\dim \D(G)} (-1)^i HP(\cR/\cJ[G]_i,d)\\[5 pt] 
&=\sum\limits_{i=0}^{\dim \D(G)}(-1)^i\sum_{\gamma\in\D(G)_i} HP(S/J(\gamma),d).
\end{array}
\]	
Assuming $D(\A_G,\m)$ is free, generated in degrees $0,A_1,\ldots,A_k$, we also have
\[
HP(D(\A_G,\m),d)=\dbinom{d+n-1}{n-1}+\sum\limits_{i=1}^k \dbinom{d+n-1-A_i}{n-1}.
\]
Equating the leading coefficients of these two expressions yields $k=n$.  Equating second coefficients yields the well-known expression $A_1+\cdots+A_n=|\m|$, where $|\m|=\sum_{ij}m_{ij}$.  Equating coefficients of $d^{n-3}$ yields the equality of so-called second local and global mixed products, $GMP(2)=LMP(2)$, defined in~\cite{TeraoCharPoly}.

This gives some insight into how a better understanding of the homologies of $\cR/\cJ[G]$ will lead to more precise obstructions to freeness.  Indeed, the Hilbert polynomial takes no account of graded dimensions that eventually vanish, while freeness may depend heavily on such information.  It is this Artinian information that we now characterize.
\end{remark}

\subsection{Freeness via syzygies}
%\section{A characterization of free multiplicities on the $A_3$ arrangement}
For the remainder of the paper, we specialize to the $A_3$ braid arrangement.  In this section we characterize free multiplicities on $A_3$ in Theorem~\ref{thm:FreeEquiv} as multiplicities for which a certain syzygy module is generated locally.
%Throughout this section, $G$ will denote the complete graph on four vertices (so $\A_G$ is the $A_3$ arrangement).  
We label $K_4$ as in Figure~\ref{fig:K4}.  Just as in Remark~\ref{rem:essA3}, we choose variables $x=x_1-x_0,y=x_2-x_0,z=x_3-x_0$.

\begin{figure}[htp]
	\centering
	\begin{tikzpicture}[scale=1]
	\tikzstyle{dot}=[circle,fill=black,inner sep=1 pt];
	
	\node[dot] (0) at (0,0) {};
	\node[dot] (1) at (0,2){};
	\node[dot] (2) at (2,-1){};
	\node[dot] (3) at (-2,-1){};
	%\node at (-2/3,1/3){$\sigma_1$};
	%\node at (0,-2/3){$\sigma_2$};
	
	\draw (0)node[below]{$v_0$}--node[right]{$x^a$}(1) node[above]{$v_1$}-- node[left]{$(z-x)^e$}(3) node[left]{$v_3$};
	\draw (0) --node[above]{$z^c$} (3) -- node[below]{$(z-y)^f$}(2)node[right]{$v_2$}-- node[above]{$y^b$} (0);
	\draw (1)--node[right]{$(y-x)^d$} (2);
	\end{tikzpicture}
	\caption{Complete graph on four vertices}\label{fig:K4}
\end{figure}
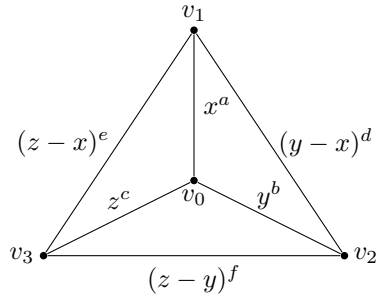
	
\begin{lem}
For any multiplicity $\m=(a,b,c,d,e,f)$, $D(\A_{K_4},\m)$ is free if and only if $H^2(\cJ[K_4])=0$.
\end{lem}
\begin{proof}
Since the clique complex $\D(K_4)$ is a three-dimensional simplex, it is contractible and $H^i(\cR[K_4])=0$ except when $i=0$.  From the long exact sequence in homology associated to
\[
0\rightarrow \cJ[K_4] \rightarrow \cR[K_4] \rightarrow \cR/\cJ[K_4]\rightarrow 0,
\]
we conclude that $H^i(\cR/\cJ[K_4])\cong H^{i+1}(\cJ[K_4])$ for $i\ge 1$.  It follows from Theorem~\ref{thm:free1} that $(\A_{K_4},\m)$ is free if and only if $H^i(\cJ[K_4])=0$ for all $i>1$.  
The complex $\cJ[K_4]$ has the form
\[
0\rightarrow \bigoplus_{ij\in\D(K_4)_1} J(ij)\rightarrow \bigoplus_{ijk\in\D(K_4)_2} J(ijk) \rightarrow J(0123) \rightarrow 0.
\]
The final map is clearly surjective, so $H^3(\cJ[K_4])=0$. Hence $(\A_{K_4},\m)$ is free if and only if $H^2(\cJ[K_4])=0$.
\end{proof}

The following lemma gives a presentation for the homology module $H^2(\cJ[K_4])$.

\begin{lem}\label{lem:H1Jpres}
Let $K_4$ have multiplicities $\m(\tau)\in\Z_+$ for each edge $\tau\in E(K_4)$. Endow the formal symbols $[e_{\tau}]$ with degrees $\m(\tau)$. We define the module of locally generated syzygies $K \subset \bigoplus\limits_{\tau\in E(K_4)} Se_\tau$ as follows. For each $\sigma \in \Delta(K_4)_2$, set 
%$K \subset \bigoplus\limits_{\tau\in E(K_4)} Se_\tau$ to be
	\[
	K_{\sigma}=\left\{\sum_{\tau\subset\sigma} a_{\tau} [e_{\tau}]: \sum a_{\tau} \alpha_{\tau}^{\m(\tau)}=0\right\},
	\]
	and $K=\sum_{\sigma} K_{\sigma}$.
	Also define the global syzygy module $V\subset \bigoplus\limits_{\tau\in E(K_4)} S[e_\tau]$ by
	\[
	V=\left\{\sum_{\tau\in E(K_4)} a_{\tau} [e_{\tau}]: \sum a_{\tau} \alpha_{\tau}^{\m(\tau)}=0\right\}.
	\]
	Then $K\subset V$ and $H^2(\cJ[K_4])\cong V/K$ as $S$-modules.
\end{lem}
\begin{proof}
	The proof is very similar to the proof of ~\cite[Lemma~3.8]{LCoho}. % Let $K_\sigma\subset \bigoplus_{\tau\subset\sigma} S[e_{\tau}]$ be the module of relations around the face $\sigma\in\D(K_4)_2$, namely
%	\[
%	K_\sigma=\left\{\sum_{\tau\subset\sigma} a_{\tau} [e_{\tau}]: \sum a_{\tau} \alpha_{\tau}^{\m(\tau)}=0\right\}.
%	\]
%	Furthermore, let $\gamma\in\D(K_4)_3$ denote the full simplex, so $J(\gamma)=\sum_{\tau\in E(K_4)} J(\tau)$.  
	Set up the following diagram with exact columns, whose first row is the complex $\cJ[K_4]$.
	
	\begin{tikzcd}
		0 & 0 & 0\\
		\bigoplus\limits_{\tau\in E(K_4)} J(\tau) \ar{r}\ar{u} & \bigoplus\limits_{\sigma\in\D(K_4)_2} J(\sigma) \ar{r}\ar{u} & J(0123) \ar{u}\\
		\bigoplus\limits_{\tau\in E(K_4)} S[e_{\tau}] \ar{r}\ar{u} & \bigoplus\limits_{\sigma \in \Delta(K_4)_2} \left[\bigoplus\limits_{\substack{\tau\in E(K_4)\\ \tau\subset \sigma}} S[e_{\tau,\sigma}] \right] \ar{r}\ar{u} & \bigoplus\limits_{\tau\in E(K_4)} S[e_{\tau}]\ar{u} \\
		0\ar{u}\ar{r} & \bigoplus\limits_{\sigma\in\D(K_4)_2} K_\sigma \ar{u}\ar{r}{\iota} & V \ar{u}\\
		& 0\ar{u} & 0\ar{u} \\
	\end{tikzcd}
	
	The middle row is in fact exact.  We argue this as follows.  Given $\tau\in\Delta(K_4)_1$, let $\Delta_\tau$ be the sub-complex of $\Delta(K_4)$ consisting of simplices which don't contain $\tau$; $\Delta_\tau$ is the union of two triangles joined along the one edge which does not intersect $\tau$.  The middle row splits as a direct sum of sub-complexes of the form
	\[
	S[e_\tau]\rightarrow S[e_{\tau,\sigma_1}]\oplus S[e_{\tau,\sigma_2}]\rightarrow S[e_\tau],
	\]
	where $\sigma_1,\sigma_2$ are the two triangles which meet along $\tau$.  The (co)homology of each of these sub-complexes may be identified with the simplicial cohomology of $\Delta(K_4)$ relative to $\Delta_\tau$, which vanishes in all dimensions.
	
	Now the long exact sequence in homology yields the isomorphisms $H^1(\cJ[K_4])\cong \mbox{ker}(\iota)$ and $H^2(\cJ[K_4])\cong \mbox{coker}(\iota)$.  The image of $\bigoplus\limits_{\sigma\in\D(K_4)_2} K_\sigma$ under $\iota$ is precisely $K$, so we are done.
\end{proof}

As a consequence of Theorem~\ref{thm:free1} and Lemma~\ref{lem:H1Jpres}, the multiplicity $\m$ is free if and only if the syzygy module of $J(0123)$ is ``locally generated,'' as we summarize in the next theorem.

\begin{thm}\label{thm:FreeEquiv}
The multiplicity $\m$ is free on $K_4$ if and only if the syzygies on the ideal $J(0123)$ are generated by the syzygies on the four sub-ideals $J(012)$, $J(013)$, $J(023)$, $J(123)$.  With notation as in Figure~\ref{fig:K4}, the multiplicity $\m=(a,b,c,d,e,f)$ is free if and only if the syzygies on
\[
\langle x^a,y^b,z^c,(y-x)^d, (z-x)^e, (z-y)^f \rangle
\]
are generated by the syzygies on the four sub-ideals
\[
\begin{array}{cc}
\langle x^a,y^b,(y-x)^d \rangle & \langle x^a, z^c, (z-x)^e \rangle \\[10 pt]
\langle y^b, z^c, (z-y)^d \rangle & \langle (z-y)^d,(z-x)^e,(y-x)^f\rangle.
\end{array}
\]
\end{thm}

\begin{cor}\label{cor:FirstFreeCondition}
Let $K_4$ be labeled as in Figure~\ref{fig:K4}.  If $J(0123)$ is minimally generated by three of the six powers $x^a$, $y^b$, $z^c$, $(y-x)^d$, $(z-x)^e$, $(z-y)^f$, and these three correspond to all the edges adjacent to a single vertex, then $D(\A_{K_4},\m)$ is free.  Up to relabeling the vertices, this may be expressed by the three simultaneous inequalities $a+b\le d+1$, $a+c\le e+1, b+c\le f+1$.
\end{cor}

\begin{remark}
Corollary~\ref{cor:FirstFreeCondition} appears in~\cite[Corollary~5.12]{EulerMult}, where the multi-arrangements $(A_3,\m)$ with these multiplicities are additionally identified as \textit{inductively free} multi-arrangements.
\end{remark}

\begin{proof}
In this case, the module $\syz(J(0123))$ is generated by three Koszul syzygies and three relations of degree $d,e,f$, expressing $(y+z)^d$, $(x+z)^e$, $(x-y)^f$ in terms of $x^a$, $y^b$, $z^c$.  Each of the modules $\syz(J(012)),\syz(J(013)),\syz(J(023))$ contributes a Koszul syzygy and one of the syzygies of degree $d,e,f$, respectively.  Hence the syzygies on $J(0123)$ are generated by the syzygies on the four sub-ideals.  The result follows from Theorem~\ref{thm:FreeEquiv}.
\end{proof}

\begin{defn}\label{defn:FreeVertex}
We call a vertex $i$ satisfying the three inequalities $m_{jk}\ge m_{ij}+m_{ik}-1$ of Corollary~\ref{cor:FirstFreeCondition} a \textit{free vertex}; if one of $0,1,2,3$ is a free vertex then we say that $\m$ has a free vertex.
\end{defn}

\section{classification, part I}\label{sec:PartI}

In this section we prove the classification of Theorem~\ref{thm:intro} for multiplicities satisfying the inequalities $m_{ij}\le m_{ik}+m_{jk}+1$ for all choices of $i,j,k$ (giving a total of 12 irredundant inequalities).  The reason for imposing these inequalities is detailed at the beginning of Section~\ref{sec:disc}; briefly, these place restrictions on the degrees in which the syzygy modules $\syz(J(ijk))$ are generated.  The remaining multiplicities are considered in Section~\ref{sec:PartII}.  Sections~\ref{sec:PartI} and ~\ref{sec:PartII} taken together constitute the proof of Theorem~\ref{thm:intro}.

\subsection{Non-free $A_3$ multiplicities via Hilbert function evaluation}\label{sec:hf}
By Theorem~\ref{thm:FreeEquiv}, we may establish that the multiplicity $(A_3,\m)$ is not free by exhibiting a degree $d$ in which the Hilbert functions of $\sum \syz(J(ijk))$ and $\syz(J(0123))$ differ.  In general it may be quite difficult to determine these Hilbert functions; however, we are able to obtain bounds. Throughout, we adopt the convention that $\binom{A}{B}=0$ if $A<B$.

We begin by describing a lower bound on the global syzygies. From the exact sequence \[
0\rightarrow \syz(J(0123))\rightarrow \bigoplus_{i,j} S(-m_{ij})\rightarrow S\rightarrow S/J(0123)\rightarrow 0,
\]
we have \[
HF(\syz(J(0123)) + HF(S) = \left[\sum_{i,j} HF(S(-m_{ij}))\right] + HF(S/J(0123)).
\]
Computing the Hilbert function of the module $S/J(0123)$ is difficult, so we settle for the following inequality.

\begin{prop}\label{prop:HFGlobal}
For all $d$, \[
HF(\syz(J(0123),d) \ge \left[\sum_{i,j} \binom{(d-m_{ij})+2}{2}\right] - \binom{d+2}{2}.
\]
\end{prop}

\begin{proof}
From above, \[
HF(\syz(J(0123)) \ge \left[\sum_{i,j} HF(S(-m_{ij})) \right] - HF(S).
\] Evaluating the right-hand side at $d$ gives the desired inequality.
\end{proof}

\begin{remark}
The bound in Proposition~\ref{prop:HFGlobal} can be improved (possibly made exact) by using inverse systems~\cite{IarrobinoI} to evaluate $\dim \syz(J(0123))_d$ exactly via a fat point computation.  This translates the ideal $J(0123)$ into a fat point ideal whose base locus is six points (corresponding to the six edges of $K_4$); these points are the intersection points of four generic lines (corresponding to the four triangles of $K_4$).  A complete classification of fat point ideals on six points, including their Hilbert function and minimal free resolution, appears in~\cite{SixFatPoints}.  Surprisingly, the weaker bound of Proposition~\ref{prop:HFGlobal} suffices for the classification of free multiplicities.
\end{remark}

Now we turn our attention to the local syzygies. The Hilbert functions of the syzygies on the individual $J(ijk)$ provide an upper bound on the Hilbert function of the local syzygy module:

\begin{prop}\label{prop:HFLocal}
\[
HF\left(\sum_{i,j,k} J(ijk) \right) \le \sum_{i,j,k} HF\left(J(ijk)\right).
\]
\end{prop}

%The syzygies on the ideals $J(ijk)$ associated to a triangle of $G=K_4$ are well understood.  Indeed, the minimal free resolution of any codimension two ideal generated by powers of linear forms is derived in~\cite{FatPoints}.

The computation of the Hilbert functions of the individual local syzygy modules is more technical and is done by Schenck \cite{FatPoints}, which we cite below in Lemma~\ref{lem:TriangleSyzygies}.

\begin{remark}
Our intuition for Schenck's result below is the following. Observe that $J(012)$ is isomorphic to $\langle x^a, y^b, (y-x)^c \rangle$. We study $\kk[x,y]/\langle x^a, y^b, (y-x)^c \rangle$, which is isomorphic to \[\left.\left(\frac{\kk[x,y]}{\langle x^a,y^b\rangle}\right)\right/\langle(y-x)^c\rangle.\] Lemma~\ref{lem:TriangleSyzygies} is equivalent to the statement that in this quotient ring, $(y-x)^c$ is a Lefschetz element (i.e., multiplication by this element is either injective or surjective). The Hilbert function increases as the degree decreases from the socle degree ($a+b-2$) to $\lceil(a+b-2)/2\rceil$.  On the other hand, since $(y-x)^c$ is a Lefschetz element, the Hilbert function of the ideal $\langle(y-x)^c\rangle$ in this quotient ring is 1 in degree $c$ and increases with the degree as long as possible. By the Hilbert-Burch Theorem, there are two minimal first syzygies. Their degrees are where the ideal's Hilbert function would exceed that of the ring. Unfortunately, these degrees depend on the parity of $a$, $b$, and $c$. The two mysterious quantities in the statement of Lemma~\ref{lem:TriangleSyzygies}, $\Omega_{ijk}$ and $a_{ijk}$, encode the parity cases simultaneously.
\end{remark}

The following lemma is an immediate consequence of \cite[Theorem~2.7]{FatPoints}.

\begin{lem}\label{lem:TriangleSyzygies}
Let $J(ijk)=\langle (x_i-x_j)^{m_{ij}},(x_i-x_k)^{m_{ik}},(x_j-x_k)^{m_{jk}} \rangle\subset S$.  Set
\[
\Omega_{ijk}=\left\lfloor\dfrac{m_{ij}+m_{jk}+m_{ik}-3}{2}\right\rfloor+1
\]
and $a_{ijk}=m_{ij}+m_{jk}+m_{ik}-2\Omega_{ijk}$.  Then, if $(x_i-x_j)^{m_{ij}}$, $(x_i-x_k)^{m_{ik}}$, and $(x_j-x_k)^{m_{jk}}$ are a minimal generating set,
\[
\syz(J(ijk))\cong S(-\Omega_{ijk}-1)^{a_{ijk}}\oplus S(-\Omega_{ijk})^{2-a_{ijk}}.
\]
Otherwise, suppose without loss of generality that $m_{ij}+m_{jk}\le m_{ik}+1$.  Then
\[
\syz(J(ijk))\cong S(-m_{ik})\oplus S(-m_{ij}-m_{jk}).
\]
\end{lem}

\begin{remark}\label{rem:ExponentFudge}
We remark for later use that if $m_{ij}\le m_{ik}+m_{jk}+1$ for all $i,j,k$ then
\[
\syz(J(ijk))\cong S(-\Omega_{ijk}-1)^{a_{ijk}}\oplus S(-\Omega_{ijk})^{2-a_{ijk}},
\]
in other words, even if $(x_i-x_j)^{m_{ij}}$, $(x_i-x_k)^{m_{ik}}$, and $(x_j-x_k)^{m_{jk}}$ are not quite a minimal generating set for $J(ijk)$, the Betti numbers for $\syz(J(ijk))$ are the same as if they were.
\end{remark}

\begin{proof}
If $(x_i-x_j)^{m_{ij}},(x_i-x_k)^{m_{ik}},(x_j-x_k)^{m_{jk}}$ are a minimal generating set, then the minimal free resolution of $J(ijk)$ has the form
\[ 
0\rightarrow S(-\Omega_{ijk}-1)^{a_{ijk}} \oplus S(-\Omega_{ijk})^{2-a_{ijk}} \xrightarrow{\phi} S(-m_{ij})\oplus S(-m_{ik})\oplus S(-m_{jk})
\]
by~\cite[Theorem~2.7]{FatPoints}.  Otherwise, if $m_{ij}+m_{jk}\le m_{ik}+1$ then $J(ijk)$ is generated by $(x_i-x_j)^{m_{ij}},(x_j-x_k)^{m_{jk}}$.  So the syzygies on the generators $(x_i-x_j)^{m_{ij}},(x_i-x_k)^{m_{ik}},(x_j-x_k)^{m_{jk}}$ are given by the Koszul syzygy on $(x_i-x_j)^{m_{ij}},(x_j-x_k)^{m_{jk}}$ and a syzygy of degree $m_{ik}$ (expressing $(x_i-x_k)^{m_{ik}}$ as a polynomial combination of $(x_i-x_j)^{m_{ij}},(x_j-x_k)^{m_{jk}}$).  See Remark~\ref{rem:syzygies}.
\end{proof}

\begin{remark}
	Since the module $\syz J(012)$ can be identified with the non-trivial derivations on the multi-arrangement $(A_2,\m)=(\A_{K_3},\m)$ (see Example~\ref{ex:ThreeCycleHomologies}), Lemma~\ref{lem:TriangleSyzygies} also follows from a result of Wakamiko~\cite{Wakamiko} on the exponents of the multi-arrangement $(A_2,\m)$.
\end{remark}

Combining the local and global bounds above, we produce a criterion for non-freeness of the multi-arrangement $(A_3,\m)$.  Define the function $LB(\m,d)$ by
\[
\begin{array}{rl}
LB(\m,d)= & \left[\sum\limits_{i,j}\dbinom{d+2-m_{ij}}{2}\right]-\dbinom{d+2}{2}-\sum\limits_{i,j,k} HF(\syz(J(ijk)),d)\\
=& 3\dbinom{d+2}{2}-\left[\sum\limits_{i,j}\dbinom{d+2-m_{ij}}{2}\right]-\left[\sum\limits_{i,j,k} HF(S/J(ijk),d)\right].
\end{array}
\]
The two different expressions for $LB(\m,d)$ are the same; this is immediate from the exact sequence
\[
0 \rightarrow \syz(J(ijk))\rightarrow S(-m_{ij})\oplus S(-m_{ik}) \oplus S(-m_{jk}) \rightarrow S \rightarrow S/J(ijk) \rightarrow 0,
\]
which holds for each $i,j,k$.

\begin{thm}\label{thm:HFBounds}
We have 
\[
HF(\syz J(0123),d) - HF\left(\sum_{i,j,k} \syz J(ijk),d \right) \ge LB(\m,d).
\] 
In particular, if $LB(\m,d)>0$ for any integer $d\ge 0$, then $(A_3,\m)$ is not free.
\end{thm}

\begin{proof}
  The inequality 
\[HF(\syz J(0123),d) - HF\left(\sum_{i,j,k} \syz J(ijk),d \right) \ge LB(\m,d)\]
follows immediately from Propositions~\ref{prop:HFGlobal} and~\ref{prop:HFLocal}.  By Theorem~\ref{thm:FreeEquiv}, $\m$ is a free multiplicity on $A_3$ if and only if
\[
HF(\syz J(0123),d) - HF\left(\sum_{i,j,k} \syz J(ijk),d \right)=0
\]
for all $d\ge 0$.
\end{proof}

\subsection{Non-free multiplicities via discriminant} \label{sec:disc}

The function $LB(\m,d)$ from Theorem~\ref{thm:HFBounds} is eventually polynomial in $d$. Denote the Hilbert polynomial by $\widetilde{LB}(\m,d)$; this is quadratic with leading coefficient $-3/2$.
%It is straightforward to see that $\widetilde{LB}(\m,d)$ is a polynomial of degree two with leading coefficient $-3/2$.  

In this section we assume that all of the ideals $J(ijk)$ are `close to' minimally generated by their three generators.  Explicitly, we impose the inequalities $m_{ij}\le m_{ik}+m_{jk}+1$ for all choices of $i,j,k$ (giving a total of 12 irredundant inequalities).  Forllowing Remark~\ref{rem:ExponentFudge} it is straightforward to check that under these assumptions, $\syz(J(ijk))$ is generated in degrees $\Omega_{ijk}+1,\Omega_{ijk}+1$ if $m_{ij}+m_{ik}+m_{jk}$ is even and degrees $\Omega_{ijk},\Omega_{ijk}+1$ if $m_{ij}+m_{ik}+m_{jk}$ is odd, where the constants $\Omega_{ijk}$ are as in Lemma~\ref{lem:TriangleSyzygies}.  Set $I=\{0,1,2,3\}$.  We have

\begin{multline*}
LB(\m,d)=\left[\sum\limits_{\{i,j\}\subset I} \dbinom{d+2-m_{ij}}{2}\right]-\dbinom{d+2}{2}\\
-\sum\limits_{\{i,j,k\}\subset I} \left(\dbinom{d+1-\Omega_{ijk}}{2}+\dbinom{d+2-\Omega_{ijk}}{2}\right).
\end{multline*}

\begin{lem}\label{lem:QuadraticMaximum}
Let $|\m|=\sum m_{ij}$.  The polynomial $\widetilde{LB}(\m,d)$ attains its maximum value at
\[
d_{max}=\frac{1}{6}\left(2|\m|-9\right).
\]
Furthermore, assume $\m$ does not have a free vertex.  Then $LB(\m,d)=\widetilde{LB}(\m,d)$ for $d\ge \lfloor d_{max}\rfloor$.
\end{lem}
\begin{proof}
Using the second expression for $LB(\m,d)$ (just prior to Theorem~\ref{thm:HFBounds}) and expanding the binomial coefficients as polynomials in $d$, we see that $\widetilde{LB}(\m,d)$ is a quadratic polynomial $Ad^2+Bd+C$ with
\begin{itemize}
	\item $A=-3/2$
	\item $B=-9/2+|\m|$
	\item $C=3-\sum_{ij}\binom{m_{ij}-1}{2}-\sum_{ijk} HP(S/J(ijk),d)$,
\end{itemize}
where $HP(S/J(ijk),d)$ is the Hilbert polynomial of $S/J(ijk)$ (since $S/J(ijk)$ is zero-dimensional as a scheme over $\mathbb{P}^2$, this is a constant).  It follows immediately that $\widetilde{LB}(\m,d)$ achieves its maximum at $d_{max}=(2|\m|-9)/6$.  For the second claim, it suffices to show that
\begin{enumerate}
\item $\lfloor d_{max}\rfloor \ge m_{ij}-2$ for all $i,j$, and
\item $\lfloor d_{max}\rfloor \ge \Omega_{ijk}-1$ for all $i,j,k$.
\end{enumerate}
For the first inequality, assume without loss of generality that $\{i,j\}=\{0,1\}$.  We have
\[
\begin{array}{rl}
2m_{23} & \ge 2\\
2(m_{03}+m_{13}) & \ge 2(m_{01}-1)\\
2(m_{02}+m_{12}) & \ge 2(m_{01}-1)\\
2m_{01} & \ge 2m_{01}.
\end{array}
\]
Summing down this list of inequalities yields $2|\m|\ge 6m_{01}-2$, so
\[
\lfloor d_{max}\rfloor =\left\lfloor \frac{1}{6}\left(2|\m|-9\right)\right\rfloor \ge \left\lfloor m_{01}-\frac{11}{6}\right\rfloor = m_{01}-2.
\]
For the second inequality, assume without loss of generality that $\{i,j,k\}=\{0,1,2\}$.  We have
\[
\begin{array}{rl}
2(m_{01}+m_{02}+m_{12}) & \ge 2(m_{01}+m_{02}+m_{12})\\
m_{13}+m_{03} & \ge m_{01}-1\\
m_{03}+m_{23} & \ge m_{02}-1\\
m_{13}+m_{23} & \ge m_{12}-1.
\end{array}
\]
Summing down this list we obtain $2|\m|\ge 3(m_{01}+m_{02}+m_{12})-3$.  In fact, we will show that $2|\m|\ge 3(m_{01}+m_{02}+m_{12})$.  Assume to the contrary that $2|\m|< 3(m_{01}+m_{02}+m_{12})$; then $2(m_{03}+m_{13}+m_{23})<m_{01}+m_{02}+m_{12}$.  Rearranging yields
\[
(m_{13}+m_{03}-m_{01})+(m_{03}+m_{23}-m_{02})+(m_{13}+m_{23}-m_{12})<0.
\]
According to the displayed inequalities above, each of the three parenthesized terms in the above sum is at least $-1$.  Consequently, each of these terms must be at most $1$, i.e. $m_{01}\ge m_{13}+m_{03}-1$, $m_{02}\ge m_{03}+m_{23}-1$, and $m_{12}\ge m_{13}+m_{23}-1$.  But then $3$ is a free vertex.  

So, assuming $\m$ does not have a free vertex, we have $2|\m|\ge 3(m_{01}+m_{02}+m_{12})$.  Hence
\[
\lfloor d_{max}\rfloor =\left\lfloor \frac{1}{6}\left(2|\m|-9\right)\right\rfloor\ge \left\lfloor\dfrac{m_{01}+m_{02}+m_{12}-1}{2}\right\rfloor-1=\Omega_{012}-1.
\qedhere
\]
\end{proof}

\begin{lem}\label{lem:Discriminant}
Let $D$ be the discriminant of the quadratic polynomial $\widetilde{LB}(\m,d)$ in the variable $d$.
\begin{enumerate}
\item If $D^2-9/4>0$, then $(A_3,\m)$ is not free.
\item If $|\m| \not \equiv 0 \pmod{3}$ and $D^2-1/4>0$, then $(A_3,\m)$ is not free.
\end{enumerate}
\end{lem}
\begin{proof}
We examine when the polynomial $\widetilde{LB}(\m,d)$ is positive at some integer $d>0$.  For this to happen, $\widetilde{LB}(\m,d)$ must have two real roots, say $r_1$ and $r_2$, and there must be an integer strictly between them.  Equivalently, there must be an integer in the interval $Q=(r_1,r_2)=(d_{max}-\frac{1}{2}|r_1-r_2|,d_{max}+\frac{1}{2}|r_1-r_2|)$.  From the form of $d_{max}$ given in Lemma~\ref{lem:QuadraticMaximum},
\begin{enumerate}
	\item If $|\m|\equiv 0 \pmod{3} $ then $d_{max}=N+1/2$ for some integer $N$
	\item If $|\m| \not \equiv 0 \pmod{3}$ then $d_{max}=N\pm 1/6$ for some integer $N$
\end{enumerate}
From the quadratic formula and the fact that the leading coefficient of $\widetilde{LB}(\m,d)$ is $-3/2$, we have
$(r_1-r_2)^2=4D^2/9$. Hence if $4D^2/9>1$, then $Q$ contains an integer. Moreover, if $|\m| \not \equiv 0 \pmod{3}$ and $4D^2/9>1/9$, then $Q$ also contains an integer. Now the result follows from Lemma~\ref{lem:QuadraticMaximum} and Theorem~\ref{thm:HFBounds}.
%Hence if $4D^2/9>1$, $\widetilde{LB}(\m,d)$ will always have a positive integer evaluation.  Moreover, if $|\m| \not \equiv 0 \pmod{3}$ and $4D^2/9>1/9$, $\widetilde{LB}(\m,d)$ will also have a positive integer evaluation.  
\end{proof}

\begin{remark}
In the following theorem, we performed the straightforward but tedious computations with the computer algebra system Mathematica.
\end{remark}

\begin{thm}\label{thm:Generalnon-free}
	Let
	\[
	P(\m)=(m_{01}+m_{23}-m_{02}-m_{13})^2+(m_{02}+m_{13}-m_{03}-m_{12})^2+(m_{03}+m_{12}-m_{01}-m_{23})^2
	\]
	and set $m_{ijk}=m_{ij}+m_{jk}+m_{ik}$.  Assume that $m_{ij}\le m_{ik}+m_{jk}+1$ for every $i$, $j$, and $k$.  Assume further that $\m$ does not have a free vertex.  If any of the conditions below are satisfied, then $\m$ is not a free multiplicity on $A_3=\A_{K_4}$.
	\begin{itemize}
		\item $|\m|\equiv 0 \mod 3$, none of the $m_{ijk}$ are odd, and $P(\m)>0$
		\item $|\m|\equiv 0 \mod 3$, two of the $m_{ijk}$ are odd, and $P(\m)>6$
		\item $|\m|\equiv 0 \mod 3$, four of the $m_{ijk}$ are odd, and $P(\m)>12$
		\item $|\m|\not\equiv 0 \mod 3$ and none of the $m_{ijk}$ are odd.
		\item $|\m|\not\equiv 0 \mod 3$, two of the $m_{ijk}$ are odd, and $P(\m)>2$
		\item $|\m|\not\equiv 0 \mod 3$, four of the $m_{ijk}$ are odd, and $P(\m)>8$.
	\end{itemize}
\end{thm}

\begin{remark}
The polynomial $P(\m)$ of Theorem~\ref{thm:Generalnon-free} is essentially an upper bound on the difference between $GMP(2)$ and $LMP(2)$, the second global and local mixed products introduced in \cite{TeraoCharPoly}.  Indeed, this theorem could be proved using these techniques.
%See Theorem~\ref{thm:AnyBraidNonFree} below for a generalization using this approach.
\end{remark}

\begin{proof}[Proof of Theorem~\ref{thm:Generalnon-free}]
Let $D$ be the discriminant of $\widetilde{LB}(\m,d)$.  From the proof of Lemma~\ref{lem:QuadraticMaximum}, $\widetilde{LB}(\m,d)=Ad^2+Bd+C$ with
\begin{itemize}
	\item $A=-3/2$
	\item $B=-9/2+|\m|$
	\item $C=3-\sum_{ij}\binom{m_{ij}-1}{2}-\sum_{ijk} HP(S/J(ijk),d)$.
\end{itemize}
Hence $D^2=B^2-4AC=9|\m|+|\m|^2-6\sum_{ij}m_{ij}^2+6\sum_{ijk} HP(S/J(ijk),d)$.  The polynomial $HP(S/J(ijk),d)$ is a constant, in fact,
\[
HP(S/J(ijk),d)=\dbinom{\Omega_{ijk}+1}{2}-\sum_{\{s,t\}\subset\{i,j,k\}}\dbinom{\Omega_{ijk}+1-m_{st}}{2}.
\]
Since the constant $\Omega_{ijk}$ depends on the parity of $m_{ijk}=m_{ij}+m_{ik}+m_{jk}$, the discriminant $D$ will also.  A straightforward computation now yields that $2(D^2-9/4)$ is equal to $P(\m)-3q$, where $q$ is the number of $m_{ijk}$ that are odd.  Note that $\sum_{ijk}m_{ijk}=2|\m|$, so $q$ equals zero, two, or four.  The dependence on the congruence class of $|\m|$ modulo three follows from Lemma~\ref{lem:Discriminant}.  In the case that $|\m|\not\equiv 0 \mod 3$ and none of the $m_{ijk}$ are odd, $2(D^2-1/4)=P(\m)+4$, which is always positive.  Hence we always have non-freeness in this case.
\end{proof}

\begin{defn}\label{defn:ANNMultiplicities}
Let $n_i\in\Z_{\ge 0}$ for $i=0,1,2,3$ and $\epsilon_{ij}\in\{-1,0,1\}$ for $0\le i<j\le 3$.  An ANN multiplicity on $A_3$ is a multiplicity of the form $m_{ij}=n_i+n_j+\epsilon_{ij}$.
\end{defn}

ANN multiplicities are classified as free or non-free in~\cite{AbeSignedEliminable} (not just on $A_3$ but on any braid arrangement).

\begin{prop}\label{prop:BigCellMultiplicities}
Let $\m$ be a multiplicity so that $m_{ij}\le m_{ik}+m_{jk}+1$ for every $i$, $j$, and $k$.  Then $\m$ is a free multiplicity for $A_3$ if and only if $\m$ has a free vertex or $\m$ is a free ANN multiplicity.
\end{prop}
\begin{proof}
If $\m$ has a free vertex then it is free by Corollary~\ref{cor:FirstFreeCondition}.  We now show that if any of the conditions of Theorem~\ref{thm:Generalnon-free} fail, then $\m$ is an ANN multiplicity.  We will do this by explicitly constructing non-negative integers $N_0,N_1,N_2,N_3$ and $\epsilon_{ij}\in\{-1,0,1\}$ so that $m_{ij}=N_i+N_j+\epsilon_{ij}$ for $0\le i<j\le 3$.  The main thing we have to be careful about is the non-negativity of the $N_i$.

We introduce some notation.  For a vertex $i$ of a triangle $ijk$, set $n_{i,ijk}=(m_{ij}+m_{ik}-m_{jk})/2$.  Since we assume $m_{jk}\le m_{ij}+m_{ik}+1$ for every triple $i,j,k$, it follows that $n_{i,ijk}\ge -1/2$.  Also, for a \textit{directed} four-cycle $ijst$ set $c_{ijst}=(m_{ij}-m_{js}+m_{st}-m_{it})/2$.

If all of the $m_{ijk}$ are even, then every expression $n_{i,ijk}$ is a non-negative integer.  In this case, negating Theorem~\ref{thm:Generalnon-free} means $P(\m)=0$; hence $c_{ijst}=0$ for every directed four cycle, and the expressions $n_{i,ijk}$ are independent of the triangle chosen to contain $i$ (for instance, $n_{0,012}=n_{0,023}=n_{0,013}$).  Set $N_0=n_{0,012},N_1=n_{1,012},N_2=n_{2,012},$ and $N_3=n_{3,013}$.  We have $N_i\ge 0$ and $m_{ij}=N_i+N_j$ for all $i,j$, so $\m$ is an ANN multiplicity.

Now suppose two of the $m_{ijk}$ are odd, and $P(\m)\le 6$.  Suppose without loss of generality that $m_{012}$ and $m_{023}$ are even, while $m_{013}$ and $m_{123}$ are odd.  Set $N_0=n_{0,012}, N_1=n_{1,012},N_2=n_{2,023},$ and $N_3=n_{3,023}$.  Note that, given our assumptions, all the $N_i$ are non-negative integers.
\[
\begin{array}{lll}
N_0+N_1=n_{0,012}+n_{1,012}=m_{01}\\
N_0+N_2=n_{0,012}+n_{2,023}=m_{02}+c_{0123}\\
N_0+N_3=n_{0,012}+n_{3,023}=m_{03}+c_{0123}\\
N_1+N_2=n_{1,012}+n_{2,023}=m_{12}+c_{0123}\\
N_1+N_3=n_{1,012}+n_{3,023}=m_{13}+c_{0132}+c_{0312}\\
N_2+N_3=n_{2,023}+n_{3,023}=m_{23}
\end{array}
\]
Note also that, under our assumptions, $c_{0123}$ is an integer while $c_{0132}$ and $c_{0312}$ are not.  We also have $c_{0123}+c_{0312}=c_{0132}$.  Since $P(\m)\le 6$, we have only the following possibilities:  
\begin{itemize}
	\item $c_{0123}=1,c_{0312}=-1/2,c_{0132}=1/2$
	\item $c_{0123}=-1,c_{0312}=1/2,c_{0132}=-1/2$
	\item $c_{0123}=0,c_{0312}=c_{0132}=1/2$
	\item $c_{0123}=0,c_{0312}=c_{0132}=-1/2$.
\end{itemize}
In any of the above situations, set $\epsilon_{01}=\epsilon_{23}=0$, $\epsilon_{02}=\epsilon_{03}=\epsilon_{12}=-c_{0123}$, and $\epsilon_{13}=-c_{0132}-c_{0312}$.  By the above observations, we have shown $m_{ij}=N_i+N_j+\epsilon_{ij}$ is an ANN multiplicity.

Finally, suppose all of the $m_{ijk}$ are odd and $P(\m)\le 12$.  In fact, $P(\m)$ is the sum of squares of three integers which add to zero, so inspection yields $P(\m)\le 8$.  Set $\tilde{N}_0=n_{0,012}, \tilde{N}_1=n_{1,013},\tilde{N}_2=n_{2,023},$ and $\tilde{N}_3=n_{3,123}$.  Note that, given our assumptions, all the $\tilde{N}_i$ are non-integers.  We modify them shortly.  We have
\[
\begin{array}{lll}
\tilde{N}_0+\tilde{N}_1=n_{0,012}+n_{1,013}=m_{01}+c_{0213}\\
\tilde{N}_0+\tilde{N}_2=n_{0,012}+n_{2,023}=m_{02}+c_{0123}\\
\tilde{N}_0+\tilde{N}_3=n_{0,012}+n_{3,123}=m_{03}+c_{0123}+c_{0213}\\
\tilde{N}_1+\tilde{N}_2=n_{1,013}+n_{2,023}=m_{12}+c_{0123}+c_{0213}\\
\tilde{N}_1+\tilde{N}_3=n_{1,013}+n_{3,123}=m_{13}+c_{0123}\\
\tilde{N}_2+\tilde{N}_3=n_{2,023}+n_{3,123}=m_{23}+c_{0213}.
\end{array}
\]
Under our assumptions, $c_{0123},c_{0213},$ and $c_{0231}$ are all integers.  We also have $c_{0123}+c_{0231}=c_{0213}$.  Since $P(\m)\le 8$, at most two of $c_{0123},c_{0213},$ and $c_{0231}$ can be non-zero, and all must have absolute value at most one.

First assume $c_{0123}=0$ and $c_{0231}=\pm 1$.  We have
\[
\begin{array}{lll}
\tilde{N}_0+\tilde{N}_1=n_{0,012}+n_{1,013}=m_{01}+c_{0213}\\
\tilde{N}_0+\tilde{N}_2=n_{0,012}+n_{2,023}=m_{02}\\
\tilde{N}_0+\tilde{N}_3=n_{0,012}+n_{3,123}=m_{03}+c_{0213}\\
\tilde{N}_1+\tilde{N}_2=n_{1,013}+n_{2,023}=m_{12}+c_{0213}\\
\tilde{N}_1+\tilde{N}_3=n_{1,013}+n_{3,123}=m_{13}\\
\tilde{N}_2+\tilde{N}_3=n_{2,023}+n_{3,123}=m_{23}+c_{0213}.
\end{array}
\]
Since $c_{0213}\ge -1$ and $m_{ij}\ge 1$ for all $i,j$, at most one of the $\tilde{N}_i$ is equal to $-1/2$.  Without loss, assume $\tilde{N}_0\ge-1/2$ while $N_i\ge 1/2$ for $i=1,2,3$.  Now set $N_0=\lceil\tilde{N}_0\rceil$, $N_1=\lfloor\tilde{N}_1\rfloor$, $N_2=\lceil\tilde{N}_2\rceil$, and $N_3=\lfloor\tilde{N}_3\rfloor$.  With these assumptions, we have
\[
\begin{array}{lll}
N_0+N_1= & n_{0,012}+n_{1,013} & =m_{01}+c_{0213}\\
N_0+N_2= & n_{0,012}+n_{2,023}+1 & =m_{02}+1\\
N_0+N_3= & n_{0,012}+n_{3,123} & =m_{03}+c_{0213}\\
N_1+N_2= & n_{1,013}+n_{2,023} & =m_{12}+c_{0213}\\
N_1+N_3= & n_{1,013}+n_{3,123}-1 & =m_{13}-1\\
N_2+N_3= & n_{2,023}+n_{3,123} & =m_{23}+c_{0213}.
\end{array}
\]
So $\m$ is an ANN multiplicity with $\epsilon_{01}=\epsilon_{03}=\epsilon_{12}=\epsilon_{23}=-c_{0213}$, $\epsilon_{02}=-1,$ and $\epsilon_{13}=1$.

The case $c_{0213}=0$ is symmetric to the above case.  We now consider the case $c_{0231}=0$, which implies $c_{0123}=c_{0213}$.  If $c_{0123}=0$ as well, then we again have at most one of $\tilde{N}_i$ equal to $-1/2$, and we argue that $\m$ is an ANN multiplicity in the same way as above.

Now suppose that $c_{0231}=0$ and $c_{0123}=c_{0213}=1$.  Then
\[
\begin{array}{lll}
\tilde{N}_0+\tilde{N}_1=n_{0,012}+n_{1,013}=m_{01}+1\\
\tilde{N}_0+\tilde{N}_2=n_{0,012}+n_{2,023}=m_{02}+1\\
\tilde{N}_0+\tilde{N}_3=n_{0,012}+n_{3,123}=m_{03}+2\\
\tilde{N}_1+\tilde{N}_2=n_{1,013}+n_{2,023}=m_{12}+2\\
\tilde{N}_1+\tilde{N}_3=n_{1,013}+n_{3,123}=m_{13}+1\\
\tilde{N}_2+\tilde{N}_3=n_{2,023}+n_{3,123}=m_{23}+1.
\end{array}
\]
In this case it is also clear that at most one of $\tilde{N}_i$ can equal $-1/2$.  If all $\tilde{N}_i$ are at least $1/2$, then we can take $N_i=\lfloor \tilde{N}_i \rfloor$ for $i=0,1,2,3$.  Then we will clearly have an ANN multiplicity.  Suppose then that one of the $\tilde{N}_i$ is equal to $-1/2$.  Without loss of generality we can assume that $\tilde{N}_0= -1/2$.  Using the third listed equation above, $\tilde{N}_3\ge 7/2$.  In this case we can set $N_0=\lceil \tilde{N}_0 \rceil=0,N_1=\lfloor \tilde{N}_1\rfloor, N_2=\lfloor \tilde{N}_2 \rfloor,$ and $N_3=\lfloor\tilde{N}_3\rfloor-1$, giving an ANN multiplicity.

Finally, suppose that $c_{0231}=0$ and $c_{0123}=c_{0213}=-1$.  Then
\[
\begin{array}{lll}
\tilde{N}_0+\tilde{N}_1=n_{0,012}+n_{1,013}=m_{01}-1\\
\tilde{N}_0+\tilde{N}_2=n_{0,012}+n_{2,023}=m_{02}-1\\
\tilde{N}_0+\tilde{N}_3=n_{0,012}+n_{3,123}=m_{03}-2\\
\tilde{N}_1+\tilde{N}_2=n_{1,013}+n_{2,023}=m_{12}-2\\
\tilde{N}_1+\tilde{N}_3=n_{1,013}+n_{3,123}=m_{13}-1\\
\tilde{N}_2+\tilde{N}_3=n_{2,023}+n_{3,123}=m_{23}-1.
\end{array}
\]
Set $N_i=\lceil\tilde{N}_i\rceil$ for $i=0,1,2,3$.  Then $N_i\ge 0$ for $i=0,1,2,3$ and we have an ANN multiplicity.
\end{proof}

\section{classification, Part II}\label{sec:PartII}
In this section we complete the classification of free multiplicities on $A_3$ given in Theorem~\ref{thm:intro}.  Our strategy is to show that, if we assume $\m$ has no free vertex and that the syzygies of $J(0123)$ are locally generated as required by Theorem~\ref{thm:FreeEquiv}, then we are forced to have the twelve inequalities $m_{ij}\le m_{ik}+m_{jk}+1$ for every triple $i,j,k$.  Then Proposition~\ref{prop:BigCellMultiplicities} guarantees that such a multiplicity is free if and only if it is a free ANN multiplicity.  We introduce some notation for studying the local syzygies.

\begin{notation}
Label the exponents with the letters $a$ through $f$ as in Figure~\ref{fig:k4-early}, and refer to the forms as $A=(x_{1}-x_{0})^{a}$, and
so on.  The local ideals $J(012)$, $J(013)$, $J(023)$, and $J(123)$
then have (not necessarily minimal) generating sets $\{A,B,D\}$,
$\{A,C,E\}$, $\{B,C,F\}$, and $\{D,E,F\}$.
\end{notation}

\begin{notation}
Consider the free $S$-module of rank six with basis $[A],[B],\ldots,[F]$.
A syzygy on $J(0123)$ is an expression of the form $g_{a}[A] + g_{b}[B] + g_{c}[C] +
g_{d}[D] + g_{e}[E] + g_{f}[F]$ satisfying $g_{a}A + g_{b}B + g_{c}C +
g_{d}D + g_{e}E + g_{f}F=0$.  Its \emph{support} is the set of generators with nonzero coefficient;
for example, the Koszul syzygy $A[B]-B[A]$ has support $\{A,B\}$.

We say that a syzygy is \emph{local} if its support is a subset of
$\{A,B,D\}$, 
$\{A,C,E\}$, $\{B,C,F\}$, or $\{D,E,F\}$, and \emph{locally	generated} if it is a linear combination of local syzygies.  
\end{notation}

\begin{notation}
We introduce notation, and an abuse thereof, to describe the
syzygies on the local ideal $J(012)=\langle A,B,D\rangle$.  We extend this
notation to the other triangles in the obvious way.
	
We denote the Koszul syzygy $A[B]-B[A]$ by $K_{ab}$; it has degree
$a+b$.  Similarly, the Koszul syzygies $K_{ad}$ and $K_{bd}$ have
degrees $a+d$ and $b+d$ respectively.  The support of the Koszul
syzygy $K_{ab}$ is $\{A,B\}$.  There are also syzygies with support
$\{A,B,D\}$; from Lemma~\ref{lem:TriangleSyzygies} these
have degree as low as $\frac{a+b+d-1}{2}$ (when $\{A,B,D\}$ is a minimal
generating set) and as low as $d$ if $D$ is not a minimal generator
(with obvious adjustments for symmetry).

Since many of our arguments below concern only the supports of the
syzygies, we abuse notation and refer to any syzygy with support
$\{A,B\}$ by the name $K_{ab}$.  (Thus, while $K_{ab}$ may not refer to the Koszul syzygy,
it does refer to an $S$-linear multiple, so all relevant intuition
about Koszul syzygies continues to work.)
Finally, $\mathcal{S}_{abd}$ will refer to any syzygy supported on a
subset of $\{A,B,D\}$.
\end{notation}

\begin{comment}
We begin with a lemma about minimal generation of the local ideals
that applies to all remaining cases.

\begin{lem}\label{lem:nonglobalimpliesnonlocal}
Suppose $A$ is redundant as a generator of $J(0123)$, but not as a
generator of $J(012)$ or $J(013)$.  Then $(A_3,\m)$ is not free.
\end{lem}
\begin{proof}
Since $A$ is redundant, we may write it as an $S$-linear combination
of the other generators.  This is a global syzygy of degree $a$.  On
the other hand, since $A$ is not locally redundant, the local
syzygies on $J(012)$ and $J(013)$ are both of degree strictly greater
than $a$.  In particular, no locally generated syzygy of degree
$a$ has $A$ in its support.  Now Theorem~\ref{thm:FreeEquiv} gives the result.
\end{proof}
\end{comment}

Without loss of generality, let $K_{be}$ have the least degree among the non-local Koszul syzygies $K_{af},K_{be},K_{cd}$.  We will show that if $\m$ is a free multiplicity with no free vertex and $K_{be}$ is locally generated (as it must be by Theorem~\ref{thm:FreeEquiv}), then $\m$ is a free ANN multiplicity.  To that end we make the following assumptions for the remainder of the section.

\begin{assumptions}\label{ass}\ 
\begin{enumerate}
	\item There is no free vertex.
	%\item The generating set $A,B,C,D,E,F$ is redundant.
	%\item Every redundant generator is redundant for one of its two triangles (this follows from assumption (2) and Lemma~\ref{lem:nonglobalimpliesnonlocal}).
	\item $b+e\le \min\{a+f,c+d\}$
	\item $K_{be}$ is locally generated.  That is, we may write
	\begin{equation}\tag{$\ast$}\label{eqn:Koszul}
	K_{be}=\mathcal{S}_{abd}+\mathcal{S}_{ace}+\mathcal{S}_{bcf}+\mathcal{S}_{def}.
	\end{equation}
\end{enumerate}
\end{assumptions}

\begin{lem}\label{lem:KoszulSupport}
	Given Assumptions~\ref{ass} and referring to Equation~\eqref{eqn:Koszul},
	\begin{itemize}
		\item If $\mathcal{S}_{def}$ is not supported on $E$ then $e\ge a+c-1$
		\item If $\mathcal{S}_{ace}$ is not supported on $E$ then $e\ge d+f-1$
		\item If $\mathcal{S}_{abd}$ is not supported on $B$ then $b\ge c+f-1$
		\item If $\mathcal{S}_{bcf}$ is not supported on $B$ then $b\ge a+d-1$.
	\end{itemize}
\end{lem}
\begin{proof}
We prove the first statement.  The remaining statements are proved in the same way.  Fixing coordinates, we may write $A=x^a,B=y^b,C=z^c,E=(x-z)^e$.

Observe $\mathcal{S}_{ace}=g_a[A]+g_c[C]+g_e[E]$, where $g_a,g_b,g_c\in S$.  On the one hand, $g_eE=-(g_aA+g_cC)$, so $g_e\in \left(\langle A,C \rangle:E\right)$.  On the other hand, since we assumed $\mathcal{S}_{def}$ is not supported on $E$, no other terms in Equation~\eqref{eqn:Koszul} are supported on $[E]$, so $g_e=-B$.  In particular, $B\in \left(\langle A,C \rangle:E\right)$.  In other words, $y^b\in \left(\langle x^a,z^c \rangle:(x-z)^e\right)$, so we conclude that $\left(\langle x^a,z^c \rangle:(x-z)^e\right)=\langle 1 \rangle$.  Consequently, $E\in\langle A,C\rangle$, which happens if and only if $e\ge a+c-1$.\qedhere
\end{proof}

%\begin{lemma}\label{lem:EliminableVertex}
%The syzygy module of $J(0123)$ is generated by the syzygies of $J(012),J(013),$ and $J(023)$ if and only if $m_{ij}\ge m_{0i}+m_{0j}-1$ for every pair $1\le i<j\le 3$.
%\end{lemma}
%\begin{proof}
%We prove the forward implication.  The reverse implication is clear.  Without loss of generality, suppose $i,j$ is the pair $1,3$, so $m_{0i}=m_{01}=a, m_{0j}=m_{03}=c$, and $m_{ij}=m_{13}=e$.  Consider the koszul syzygy $K_{be}$; by assumption in the local expression for $K_{be}$, $\mathcal{S}_{def}=0$.  By Lemma~\ref{lem:KoszulSupport}, $e\ge a+c-1$.
%\end{proof}

\begin{lem}\label{lem:ChoosingPair}
Given Assumptions~\ref{ass} and referring to Equation~\eqref{eqn:Koszul}, $\mathcal{S}_{ace}$ and $\mathcal{S}_{def}$ must both be supported on the edge $E$.  Likewise $\mathcal{S}_{abd}$ and $\mathcal{S}_{bcf}$ must both be supported on $B$.
\end{lem}
\begin{proof}
In light of Lemma~\ref{lem:KoszulSupport}, it suffices to show that we have the four strict inequalities $e<a+c-1, e<d+f-1, b<c+f-1,$ and $b<a+d-1$.  We show the inequality $e<a+c-1$; the rest follow by symmetry.  Suppose to the contrary that $e\ge a+c-1$.  Then, since $b+e\le\min\{a+f,c+d\}$,
\[
\begin{array}{rl}
a+f\ge b+e\ge b+a+c-1\\
c+d\ge b+e\ge b+a+c-1.
\end{array}
\]
Consequently we have $f\ge b+c-1$ and $d\ge b+a-1$.  Since we assumed $e\ge a+c-1$, we conclude that vertex $0$ is a free vertex, violating Assumption~\ref{ass}.(1).
\end{proof}

\begin{notation}
We say that an edge is in the \textit{support} of the local expression~\eqref{eqn:Koszul} for $K_{be}$ if it is in the support of one of the summands.
\end{notation}

\begin{lem}\label{lem:FullSupport}
	In the local expression
	\[
	K_{be}=\mathcal{S}_{abd}+\mathcal{S}_{ace}+\mathcal{S}_{bcf}+\mathcal{S}_{def},
	\]
	each summand must be supported on three edges.
\end{lem}
\begin{proof}
	By Lemma~\ref{lem:ChoosingPair}, we already know that each summand is supported on either $E$ or $B$.  Next we claim that the local expression for $K_{be}$ must be supported on at least three of the edges $A,C,D,$ and $F$.  Suppose to the contrary that two of these edges are absent from the support.  Up to symmetry there are two possibilities: either the two edges are adjacent ($A$ and $C$) or the two edges are opposite ($A$ and $F$).  In the first case, we have $\mathcal{S}_{ace}=0$, contradicting Lemma~\ref{lem:ChoosingPair}.  If the local expression is supported on $A$ and $F$, then this forces
	\[
	K_{be}=K_{bd}+K_{ce}+K_{bc}+K_{de},
	\]
	which is impossible due to degree considerations, as we now explain.  Looking at coefficients on $[C]$ yields $(pE+qB)[C]=0$, so $pE+qB=0$.  Since $B$ and $E$ have no common factor, $\deg(pE)\ge b+e$, so $\deg (K_{ce})\ge b+c+e>b+e$, a contradiction.  (If $p=q=0$ then $\mathcal{S}_{ace}=\mathcal{S}_{def}=0$, again contradicting Lemma~\ref{lem:ChoosingPair}.)
	
	Now suppose that the local expression for $K_{be}$ is supported on all but one of the edges $A,C,D,$ and $F$, without loss of generality the edge $A$.  Then we have the equation below.
	\[
	\begin{array}{rcrrrrrrl}
	K_{be}=&&(&E[B]& & &-B[E] & &)\\
	=&&g_{abd}(&D[B]&&-B[D]&&&)\\
	&+&g_{ace}(&&E[C]&&-C[E]&&)\\
	&+&g_{bcf}(&j_{b}[B]&+j_{c}[C]&&&j_{f}[F]&)\\
	&+&g_{def}(&&&h_{d}[D]&+h_{e}[E]&+h_{f}[F]&).
	\end{array}
	\]
	Equating coefficients on $[B]$ and inspecting degrees yields $d\le e$, while equating coefficients on $[E]$ yields $c\le b$.  Since $b+e\le c+d$, this implies $d=e$ and $c=b$.
	
	Since $c+e=b+e$, we conclude that $g_{ace}$ is a scalar, so, looking
	at the coefficients on $[C]$, we conclude that (up to scalar)
	$g_{bcf}j_{c}=E$.  Thus $\mathcal{S}_{bcf}$ is equivalent (up to
	scalar) to $EC=g_{bcf}j_{b}B+g_{bcf}j_{f}F$, and we conclude $EC\in
	\langle B,F\rangle$.  But $\langle B,F\rangle$ is a primary ideal and $E^n$ is not in $\langle B,F\rangle$ for any $n$ (since $B,E,F$ form a regular sequence), so $C\in \langle B,F\rangle$, i.e. $b+f\le c+1$.  Since $b=c$, this implies $f=1$.  But then we have the multiplicity $(a,b,b,d,d,1)$ and $b+e=b+d\le a+f=a+1$, so we have a free vertex (in fact, vertices $2$ and $3$ are both free), a contradiction.

	Now we show that each of the local syzygies is supported on all three of its edges.  It is enough to do this for $\mathcal{S}_{abd}=\alpha_{abd}[A]+\beta_{abd}[B]+\delta_{abd}[D]$.  We already know from Lemma~\ref{lem:ChoosingPair} that $\mathcal{S}_{ace}$ is supported on $B$ (i.e. $\beta_{abd}\neq 0$).  It suffices to show that $\mathcal{S}_{abd}$ is supported on $A$ (the argument for support on $D$ is the same).  Adding coefficients on $[A]$ in Equation~\eqref{eqn:Koszul} yields $\alpha_{abd}+\alpha_{ace}=0$.  If $\mathcal{S}_{abd}$ is not supported on $A$, then $\alpha_{abd}=0$, so $\alpha_{ace}=0$ as well.  Then the local expression~\eqref{eqn:Koszul} is not supported on $A$, a contradiction.
\end{proof}

We are now ready to complete the proof of Theorem~\ref{thm:intro}.

\begin{prop}\label{prop:FreeInequalities}
If $\m$ is a free multiplicity without a free vertex, then $m_{ij}\le m_{ik}+m_{jk}+1$ for every triple $i,j,k$.
\end{prop}
\begin{proof}
Theorem~\ref{thm:FreeEquiv} guarantees Assumption~\ref{ass}.(3), so we may take all of Assumptions~\ref{ass} without loss.  By the proof of Lemma~\ref{lem:ChoosingPair}, we already have (stricter versions of) the four inequalities $b\le c+f+1,b\le a+d+1,e\le a+c+1,$ and $e\le d+f+1$.  Hence we need to establish the eight remaining inequalities with $a,c,d,$ and $f$ on the left-hand side.  We demonstrate the inequality $a\le b+d+1$.  By symmetry, the remaining seven inequalities are established in precisely the same way.
	
Since $b+e\le c+d$, we have $b+e\le (b+c+d+e)/2$.  By Lemma~\ref{lem:FullSupport}, $\mathcal{S}_{ace}$ is supported on $[A],[C],$ and $[E]$.  It follows that the degree of $\mathcal{S}_{ace}$ is at least $(a+c+e-1)/2$ by Lemma~\ref{lem:TriangleSyzygies}.  Since $\mathcal{S}_{ace}$ appears in the expression for $K_{be}$, we have
\[
\dfrac{a+c+e-1}{2}\le b+e\le \dfrac{b+c+d+e}{2}.
\]
Simplifying yields $a\le b+d+1$, as desired.
\end{proof}

\begin{proof}[Proof of Theorem~\ref{thm:intro}]
Suppose $\m$ is a free multiplicity without a free vertex.  By Proposition~\ref{prop:FreeInequalities}, $m_{ij}\le m_{ik}+m_{jk}+1$ for every triple $i,j,k$.  By Proposition~\ref{prop:BigCellMultiplicities}, $\m$ must be a free ANN multiplicity.
\end{proof}

\begin{remark}\label{rem:Abe}
	A \textit{deformation} of the $A_3$ arrangement (technically, the \textit{cone} over a deformation of the $A_3$ arrangement) is a central hyperplane arrangement of the form
	\[
	\begin{array}{rl}
	x= & \alpha_1w,\ldots,\alpha_aw\\
	y= & \beta_1w,\ldots,\beta_bw\\
	z= & \kappa_1w,\ldots,\kappa_cw\\
	y-x= & \delta_1w,\ldots,\delta_dw\\
	z-x= & \epsilon_1w,\ldots,\epsilon_ew\\
	y-z= & \phi_1w,\ldots,\phi_fw\\
	w= & 0,
	\end{array}
	\]
	where $\alpha_i,\beta_i,\kappa_i,\delta_i,\epsilon_i,\phi_i$ are all elements of the ground field $\kk$.  Arrangements of this type were first investigated systematically by Stanley~\cite{StanleyIntervalOrders} and have since been the subject of many research papers.
	
	Our results may be used to show that freeness of a deformation of the $A_3$ arrangement can be detected just from its intersection lattice.  This is readily deduced from general characterizations of freeness due to Yoshinaga~\cite{YoshCharacterizationFreeArr} and Abe-Yoshinaga~\cite{AbeYoshFreeCharPoly}.  Integral to both of these characterizations is the freeness of the multi-arrangement obtained from restricting the arrangement to a chosen hyperplane, where the multiplicity assigned to each hyperplane $H$ in the restriction counts the number of hyperplanes that restrict to $H$.  In the case of a deformation of $A_3$, restricting to the hyperplane $w=0$ clearly results in the multi-arrangement $(A_3,(a,b,c,d,e,f))$; freeness of this multi-arrangement is determined from Theorem~\ref{thm:intro}.
\end{remark}

\section{Abe-Nuida-Numata multiplicities}\label{sec:ANN}

In this section we relate our results more closely to the classification of ANN multiplicities by Abe-Nuida-Numata in~\cite{AbeSignedEliminable}.  We first state their classification precisely for the $A_3$ arrangement.  We then show that the non-free multiplicities in their classification follow from Theorem~\ref{thm:Generalnon-free} and Proposition~\ref{prop:FreeInequalities}.  Finally, we illustrate how the free multiplicities in their classification may be used to provide the minimal free resolution of the ideal $J(0123)$ generated by powers of linear forms.

We introduce the notation from~\cite{AbeSignedEliminable}.  Let $G$ be a \textit{signed} graph on four vertices. That is, each edge of $G$ is assigned either a $+$ or a $-$, and so the edge set $E_G$ decomposes as a disjoint union $E_G=E_G^+\cup E_G^-$.
%$G$ is a graph each of whose edges is assigned either a $+$ or $-$; the edge set $E_G$ then decomposes as a disjoint union $E_G=E_G^+\cup E_G^-$.  
Define
\[
m_G(ij)=\left\lbrace
\begin{array}{rl}
1 & \{i,j\}\in E^+_G\\
-1 & \{i,j\}\in E^-_G\\
0 & \text{otherwise}.
\end{array}
\right.
\]

The graph $G$ is \textit{signed-eliminable} with \textit{signed-elimination ordering} $\nu:V(G)\rightarrow \{0,1,2,3\}$ if $\nu$ is bijective, and, for every three vertices $v_i,v_j,v_k\in V(G)$ with $\nu(v_i),\nu(v_j)<\nu(v_k)$, the induced subgraph $G|_{v_i,v_j,v_k}$ satisfies the following conditions.
\begin{itemize}
	\item For $\sigma\in\{+,1\}$, if $\{v_i,v_k\}$ and $\{v_j,v_k\}$ are edges in $E^\sigma_G$ then $\{v_i,v_j\}\in E^\sigma_G$
	\item For $\sigma\in\{+,1\}$, if $\{v_k,v_i\}\in E^\sigma_G$ and $\{v_i,v_j\}\in E^{-\sigma}_G$ then $\{v_k,v_j\}\in E_G$
\end{itemize} 
For a signed-eliminable graph $G$ with signed elimination ordering $\nu$, $v\in V_G$ and $i\in\{0,1,2,3\}$, define the degree $\widetilde{\deg}_i(v)$ by
\[
\widetilde{\deg}_i(v):=\deg(v,V_G,E_G^+|_{\nu^{-1}\{1,\ldots,i\}})-\deg(v,V_G,E_G^-|_{\nu^{-1}\{1,\ldots,i\}}),
\]
where $\deg(w,V_H,E_H)$ is the degree of the vertex $w$ in the graph $(V_H,E_H)$ and $(V_G,E^\sigma_G|_S)$ with respect to $S\subset V_G$ is the induced subgraph of $G$ whose edge set is $\{\{v_i,v_j\}\in E^\sigma_G \mid v_i,v_j\in S\}$.  Furthermore set $\widetilde{\deg}_i=\widetilde{\deg}_i(\nu^{-1}(i))$ for $i=0,1,2,3$.

All signed-eliminable graphs on four vertices are listed (with an elimination ordering) in~\cite[Example~2.1]{AbeSignedEliminable}, along with those which are not signed-eliminable.  For use in the proof of Corollary~\ref{cor:ANNmultiplicities}, we also list those graphs which are not signed-eliminable in Table~\ref{tbl:NonSignedEliminable}.  The property of being signed-eliminable is preserved under interchanging $+$ and $-$. Consequently, we list these graphs in Table~\ref{tbl:NonSignedEliminable} up to automorphism with the convention that a single edge takes one of the signs $+,-$, while a double edge takes the other sign.

\begin{thm}\cite[Theorem~0.3]{AbeSignedEliminable}\label{thm:ANNmultiplicities}
	Let $k$, $n_0$, $n_1$, $n_2$, and $n_3$ be nonnegative integers, and $G$ be a signed graph on four vertices.  Define the multiplicity $\m$ on the braid arrangement $A_3$ by $m_{ij}=2k+n_i+n_j+m_G(ij)$.  Set $N=4k+n_0+n_1+n_2+n_3$.  Assume one of the three conditions:
	\begin{enumerate}
		\item $k>0$
		\item $E^-_G=\emptyset$
		\item $E^+_G=\emptyset$ and $m_{ij}>0$ for every $\{i,j\}\in E_{K_4}$.
	\end{enumerate}
	Then $(A_3,\m)$ is free with exponents $(0,N+\widetilde{\deg}_2,N+\widetilde{\deg}_3,N+\widetilde{\deg}_3)$ if and only if $G$ is signed-eliminable.
\end{thm}

\begin{table}
	\begin{tabular}{cccccc}
		
		\begin{tikzpicture}[scale=1.0]
		\tikzstyle{dot}=[circle,fill=black,inner sep=1 pt];
		
		\node[dot] (1) {};
		\node[dot] (2)[left of=1]{};
		\node[dot] (3)[below of=2]{};
		\node[dot] (4)[below of=1]{};
		
		\draw (1)--(4) (3)--(2);
		\draw[double distance=1 pt] (3)--(4);
		\end{tikzpicture}
		&
		\begin{tikzpicture}[scale=1.0]
		\tikzstyle{dot}=[circle,fill=black,inner sep=1 pt];
		
		\node[dot] (1) {};
		\node[dot] (2)[left of=1]{};
		\node[dot] (3)[below of=2]{};
		\node[dot] (4)[below of=1]{};
		
		\draw (1)--(2) (1)--(3);
		\draw[double distance=1 pt] (1)--(4);
		\end{tikzpicture}
		&
		\begin{tikzpicture}[scale=1.0]
		\tikzstyle{dot}=[circle,fill=black,inner sep=1 pt];
		
		\node[dot] (1) {};
		\node[dot] (2)[left of=1]{};
		\node[dot] (3)[below of=2]{};
		\node[dot] (4)[below of=1]{};
		
		\draw (1)--(2)--(3)--(4)--(1);
		\end{tikzpicture}
		&
		\begin{tikzpicture}[scale=1.0]
		\tikzstyle{dot}=[circle,fill=black,inner sep=1 pt];
		
		\node[dot] (1) {};
		\node[dot] (2)[left of=1]{};
		\node[dot] (3)[below of=2]{};
		\node[dot] (4)[below of=1]{};
		
		\draw (4)--(1)--(2)--(3);
		\draw[double distance=1 pt] (3)--(4);
		\end{tikzpicture}
		&
		\begin{tikzpicture}[scale=1.0]
		\tikzstyle{dot}=[circle,fill=black,inner sep=1 pt];
		
		\node[dot] (1) {};
		\node[dot] (2)[left of=1]{};
		\node[dot] (3)[below of=2]{};
		\node[dot] (4)[below of=1]{};
		
		\draw (1)--(4)--(2)--(3);
		\draw[double distance=1 pt] (3)--(4);
		\end{tikzpicture}
		&	
		\begin{tikzpicture}[scale=1.0]
		\tikzstyle{dot}=[circle,fill=black,inner sep=1 pt];
		
		\node[dot] (1) {};
		\node[dot] (2)[left of=1]{};
		\node[dot] (3)[below of=2]{};
		\node[dot] (4)[below of=1]{};
		
		\draw (2)--(4)--(3);
		\draw[double distance=1 pt] (2)--(3) (1)--(4);
		\end{tikzpicture}
		\\
		\begin{tikzpicture}[scale=1.0]
		\tikzstyle{dot}=[circle,fill=black,inner sep=1 pt];
		
		\node[dot] (1) {};
		\node[dot] (2)[left of=1]{};
		\node[dot] (3)[below of=2]{};
		\node[dot] (4)[below of=1]{};
		
		\draw (1)--(4) (3)--(2);
		\draw[double distance=1 pt] (1)--(2) (3)--(4);
		\end{tikzpicture}
		&
		\begin{tikzpicture}[scale=1.0]
		\tikzstyle{dot}=[circle,fill=black,inner sep=1 pt];
		
		\node[dot] (1) {};
		\node[dot] (2)[left of=1]{};
		\node[dot] (3)[below of=2]{};
		\node[dot] (4)[below of=1]{};
		
		\draw (1)--(2)--(3)--(4)--(1);
		\draw[double distance=1 pt] (1)--(3);
		\end{tikzpicture}
		&
		\begin{tikzpicture}[scale=1.0]
		\tikzstyle{dot}=[circle,fill=black,inner sep=1 pt];
		
		\node[dot] (1) {};
		\node[dot] (2)[left of=1]{};
		\node[dot] (3)[below of=2]{};
		\node[dot] (4)[below of=1]{};
		
		\draw (1)--(2)--(3)--(4);
		\draw[double distance=1 pt] (3)--(1)--(4);
		\end{tikzpicture}
		&
		\begin{tikzpicture}[scale=1.0]
		\tikzstyle{dot}=[circle,fill=black,inner sep=1 pt];
		
		\node[dot] (1) {};
		\node[dot] (2)[left of=1]{};
		\node[dot] (3)[below of=2]{};
		\node[dot] (4)[below of=1]{};
		
		\draw (4)--(1)--(3)--(2);
		\draw[double distance=1 pt] (1)--(2) (3)--(4);
		\end{tikzpicture}
		&
		\begin{tikzpicture}[scale=1.0]
		\tikzstyle{dot}=[circle,fill=black,inner sep=1 pt];
		
		\node[dot] (1) {};
		\node[dot] (2)[left of=1]{};
		\node[dot] (3)[below of=2]{};
		\node[dot] (4)[below of=1]{};
		
		\draw (1)--(4)--(2)--(3);
		\draw[double distance=1 pt] (2)--(1)--(3)--(4);
		\end{tikzpicture}
		&
		\begin{tikzpicture}[scale=1.0]
		\tikzstyle{dot}=[circle,fill=black,inner sep=1 pt];
		
		\node[dot] (1) {};
		\node[dot] (2)[left of=1]{};
		\node[dot] (3)[below of=2]{};
		\node[dot] (4)[below of=1]{};
		
		\draw (1)--(2)--(4)--(3)--(1);
		\draw[double distance=1 pt] (2)--(3) (1)--(4);
		\end{tikzpicture}
	\end{tabular}
	
	\caption{Graphs on four vertices which are \textit{not} signed-eliminable}\label{tbl:NonSignedEliminable}
\end{table}

We first show how we can recover the non-free ANN multiplicities on $A_3$ using Theorem~\ref{thm:Generalnon-free}.

\begin{cor}\label{cor:ANNmultiplicities}
Let $k,n_0,n_1,n_2,$ and $n_3$ be non-negative integers, let $G$ be a signed graph on $K_4$, and let $\m$ be the ANN multiplicity $m_{ij}=2k+n_i+n_j+m_G(ij)$.  If one of the two following conditions is satisfied, then $\m$ is not a free multiplicity.
\begin{enumerate}
\item One or more of the inequalities $\{m_{ij}+m_{ik}+1\ge m_{jk}\mid 0\le i<j<k\le 3\}$ fails and $\m$ does not have a free vertex.
\item All of the inequalities $\{m_{ij}+m_{ik}+1\ge m_{jk}\mid 0\le i<j<k\le 3\}$ are satisfied and $G$ is not signed-eliminable.
\end{enumerate}

\end{cor}
\begin{proof}
	If the ANN multiplicity fails one or more of the inequalities $\{m_{ij}+m_{ik}+1\ge m_{jk}\mid 0\le i<j<k\le 3\}$, then it is free if and only if it has a free vertex by Proposition~\ref{prop:BigCellMultiplicities}, completing the proof of $(1)$.
	
%	So we assume that these inequalities are all satisfied.  For ANN multiplicities, this is equivalent to $2(k+n_i)+1+m_G(ij)+m_G(ik)\ge m_G(jk)$ for every triple $i,j,k$.  Suppose without loss of generality that $2(k+n_0)+1+m_G(01)+m_G(02)< m_G(12)$.  Look at the three conditions in Theorem~\ref{thm:ANNmultiplicities}, one of which must be satisfied.  Clearly $k=n_0=0$, and $E^-_G$ cannot be empty.  Hence we must have $E^+_G=\emptyset$ and $m_{ij}>0$ for all $i,j$.  Referring to Table~\ref{tbl:NonSignedEliminable}, there is only one signed graph on four vertices with all edges negative which is not signed-eliminable; namely the four-cycle with all negative edges.  The corresponding multiplicity (assuming $k=n_0=0$ and $m_G(01)=m_G(02)=-1$) is $(n_1-1,n_2-1,n_3,n_1+n_2,n_1+n_3-1,n_2+n_3-1)$, where we must also assume $n_1>1,n_2>1,$ and $n_3>0$ so that $m_{ij}>0$ for all $i,j$.  This multiplicity has no free vertex and fails the inequality $m_{01}+m_{02}+1\ge m_{12}$, hence by Proposition~\ref{prop:FreeInequalities} it is not free.
	
	We now assume the inequalities $m_{ij}+m_{ik}+1\ge m_{jk}$ on all triples $0\le i<j<k\le 3$.  We apply Theorem~\ref{thm:Generalnon-free}.  It is evident that $P(m_{ij})=P(2k+n_i+n_j+m_G(ij))=P(m_G(ij))$.  Hence it is enough to show that $P(m_G(ij))$ satisfies one of the inequalities of Theorem~\ref{thm:Generalnon-free} if $G$ is not signed-eliminable.  This can be verified on a case-by-case basis; going across Table~\ref{tbl:NonSignedEliminable} from left to right and top to bottom:
	
	\begin{itemize}
		\item Two of $m_{ijk}$ odd, $P(m_G(ij))=14>6$
		\item None of $m_{ijk}$ odd, $P(m_G(ij))=8>0$
		\item None of $m_{ijk}$ odd, $P(m_G(ij))=8>0$
		\item None of $m_{ijk}$ odd, $P(m_G(ij))=8>0$
		\item Two of $m_{ijk}$ odd, $P(m_G(ij))=14>6$
		\item Two of $m_{ijk}$ odd, $P(m_G(ij))=18>6$
		\item None of $m_{ijk}$ odd, $P(m_G(ij))=24>0$
		\item Two of $m_{ijk}$ odd, $P(m_G(ij))=18>6$
		\item Two of $m_{ijk}$ odd, $P(m_G(ij))=14>6$
		\item Two of $m_{ijk}$ odd, $P(m_G(ij))=26>6$
		\item All of $m_{ijk}$ odd, $P(m_G(ij))=24>12$
		\item All of $m_{ijk}$ odd, $P(m_G(ij))=32>12$ \qedhere 
	\end{itemize}

\end{proof}

We conclude by remarking on how to use free ANN multiplicities to construct the minimal free resolution of the ideal $S/J(0123)$.

\begin{cor}\label{cor:ThirdSyzygy}
	The multi-arrangement $(A_3,\m)$ is free if and only if $D(A_3,\m)$ is a third syzygy module of $S/J(0123)$ (in a non-minimal resolution).
\end{cor}
\begin{proof}
	If $D(A_3,\m)$ is a third syzygy module, it is free by the Hilbert Syzygy Theorem.  On the other hand, suppose $D(A_3,\m)$ is free.  Let $K=\sum_{\sigma} K_\sigma$ be as in Lemma~\ref{lem:H1Jpres} and the inclusion $\iota: \bigoplus_\sigma K_{\sigma}\rightarrow V$ be as in the diagram in the proof of Lemma~\ref{lem:H1Jpres}.  Consider the chain complex
	\[
	0\rightarrow \ker(\iota) \rightarrow \bigoplus\limits_{\sigma\in \D(K_4)_2} K_\sigma \xrightarrow{\iota} \bigoplus\limits_{\tau\in E(K_4)} Se_\tau \rightarrow S \rightarrow S/J(0123) \rightarrow 0
	\]
	Since $D(A_3,\m)$ is free, the above complex is exact by Theorem~\ref{thm:FreeEquiv}.  The modules $K_\sigma$, being syzygy modules of codimension two ideals, are free modules.  The long exact sequence in homology applied to the diagram in the proof of Lemma~\ref{lem:H1Jpres} yields that $H^1(\cJ[K_4])\cong \ker(\iota)$.  Since $D(A_3,\m)\cong H^0(\cR/\cJ[K_4])\cong S\oplus H^1(\cJ[K_4])$, $D(A_3,\m)$ is a (non-minimal) third syzygy of $S/J(0123)$.
\end{proof}

\begin{remark}
While the minimal free resolution of an ideal in two variables generated by powers of linear forms is known (see~\cite{FatPoints}), there is relatively little known about minimal free resolutions of ideals generated by powers of linear forms in three variables.  See~\cite[Conjecture~6.3]{HalLinearSystems} for a conjecture on the minimal free resolution for an ideal generated by powers of seven linear forms in three variables.
\end{remark}

Using Corollary~\ref{cor:ThirdSyzygy}, we can use the result of Abe-Nuida-Numata to construct the minimal free resolution of $J(0123)$ whenever $\m$ is a free ANN multiplicity.

\begin{cor}\label{cor:mfr}
Let $G$ be a signed-eliminable graph on four vertices with signed-elimination ordering $\nu$.  Let $k,n_0,n_1,n_2,n_3$ be nonnegative integers and $\m$ be the multiplicity on $A_3$ with $m_{ij}=2k+n_i+n_j+m_G(ij)$.  Also set $N=4k+(n_0+n_1+n_2+n_3)$, and let $\Omega_{ijk}$ be as in Lemma~\ref{lem:TriangleSyzygies}.  Then the ideal $J(0123)=\langle (x_i-x_j)^{m_{ij}} | 0\le i<j\le 3\rangle$ has free resolution:
\[
0\rightarrow \bigoplus\limits_{i=1}^3 S(-N-\widetilde{deg}_i) \rightarrow \bigoplus\limits_{i,j,k}\left( S(-\Omega_{ijk})^{a_{ijk}}\oplus S(-\Omega_{ijk}-1)^{2-a_{ijk}} \right) \rightarrow \bigoplus\limits_{i,j} S(-m_{ij}) \rightarrow J(0123)
\]
Furthermore, if none of the six generators are redundant, this resolution is minimal.
\end{cor}

We describe three special cases of Corollary~\ref{cor:mfr}.  If $\m$ is constant with $m_{ij}=2k$, then
\[
0\rightarrow S(-4k)^3\rightarrow S(-3k)^8 \rightarrow S(-2k)^6 \rightarrow S
\]
is a minimal free resolution for $S/J(0123)$.  If $\m$ is constant with $m_{ij}=2k+1$, then to use Corollary~\ref{cor:mfr} we take $G$ to be the complete graph on four vertices with all edges signed positively.  Then $\widetilde{deg}_2=1,\widetilde{deg}_2=2,$ and $\widetilde{deg}_3=3$.  Hence
\[
0\rightarrow S(-4k-1) \oplus S(-4k-2) \oplus S(-4k-3) \rightarrow S(-3k-1)^4\oplus S(-3k-2)^4 \rightarrow S(-2k-1)^6 \rightarrow S
\]
is a minimal free resolution for $S/J(0123)$.  Finally, suppose that $m_{ij}=n_i+n_j$ for positive integers $n_0,n_1,n_2,n_3$.  Then
\[
0\rightarrow S(-\sum n_i)^3\rightarrow \bigoplus\limits_{ijk} S(-n_i-n_j-n_k)^2 \rightarrow \bigoplus\limits_{i,j} S(-n_i-n_j) \rightarrow S
\]
is a minimal free resolution for $S/J(0123)$.

\noindent \textbf{Acknowledgments}: We are especially grateful to Takuro Abe for helpful comments and for pointing out the application to deformations of $A_3$, which we have included in Remark~\ref{rem:Abe}.  We are further grateful to Hal Schenck, Alexandra Seceleanu, Jianyun Shan, and Max Wakefield for feedback on earlier drafts of the paper.  We used Macaulay 2 \cite{M2} and Mathematica \cite{Mathematica} in our computations. Some of the computing for this project was performed at the OSU High Performance Computing Center at Oklahoma State University supported in part through the National Science Foundation grant ACI–1126330. The work in this paper was partially supported by grants from the Simons Foundation (\#199124 to Christopher Francisco and \#202115 to Jeffrey Mermin). 

\bibliography{GraphBib,SplinesBib}{}

\begin{thebibliography}{{DiP}16}

\bibitem[Abe07]{AbeDeletedA3}
Takuro Abe.
\newblock Free and non-free multiplicity on the deleted {$A_3$} arrangement.
\newblock {\em Proc. Japan Acad. Ser. A Math. Sci.}, 83(7):99--103, 2007.

\bibitem[ANN09]{AbeSignedEliminable}
Takuro Abe, Koji Nuida, and Yasuhide Numata.
\newblock Signed-eliminable graphs and free multiplicities on the braid
  arrangement.
\newblock {\em J. Lond. Math. Soc. (2)}, 80(1):121--134, 2009.

\bibitem[ATW07]{TeraoCharPoly}
Takuro Abe, Hiroaki Terao, and Max Wakefield.
\newblock The characteristic polynomial of a multiarrangement.
\newblock {\em Adv. Math.}, 215(2):825--838, 2007.

\bibitem[ATW08]{EulerMult}
Takuro Abe, Hiroaki Terao, and Max Wakefield.
\newblock The {E}uler multiplicity and addition-deletion theorems for
  multiarrangements.
\newblock {\em J. Lond. Math. Soc. (2)}, 77(2):335--348, 2008.

\bibitem[AY09]{AbeQuasiConstant}
Takuro Abe and Masahiko Yoshinaga.
\newblock Coxeter multiarrangements with quasi-constant multiplicities.
\newblock {\em J. Algebra}, 322(8):2839--2847, 2009.

\bibitem[AY13]{AbeYoshFreeCharPoly}
Takuro Abe and Masahiko Yoshinaga.
\newblock Free arrangements and coefficients of characteristic polynomials.
\newblock {\em Math. Z.}, 275(3-4):911--919, 2013.

\bibitem[{DiP}16]{GSplinesGraphicArrangements}
M.~{DiPasquale}.
\newblock {Generalized Splines and Graphic Arrangements}.
\newblock {\em J. Algebraic Combin.}, 2016.
\newblock
  \mbox{doi}:\href{http://dx.doi.org/10.1007/s10801-016-0704-8}{10.1007/s10801-016-0704-8}.

\bibitem[EI95]{IarrobinoI}
J.~Emsalem and A.~Iarrobino.
\newblock Inverse system of a symbolic power. {I}.
\newblock {\em J. Algebra}, 174(3):1080--1090, 1995.

\bibitem[GH07]{SixFatPoints}
Elena Guardo and Brian Harbourne.
\newblock Resolutions of ideals of any six fat points in {${\bf P}^2$}.
\newblock {\em J. Algebra}, 318(2):619--640, 2007.

\bibitem[GS]{M2}
Daniel~R. Grayson and Michael~E. Stillman.
\newblock Macaulay2, a software system for research in algebraic geometry.
\newblock Available at \href{http://www.math.uiuc.edu/Macaulay2/}%
  {http://www.math.uiuc.edu/Macaulay2/}.

\bibitem[GS98]{FatPoints}
Anthony~V. Geramita and Henry~K. Schenck.
\newblock Fat points, inverse systems, and piecewise polynomial functions.
\newblock {\em J. Algebra}, 204(1):116--128, 1998.

\bibitem[OT92]{OrlikTerao}
Peter Orlik and Hiroaki Terao.
\newblock {\em Arrangements of hyperplanes}, volume 300 of {\em Grundlehren der
  Mathematischen Wissenschaften [Fundamental Principles of Mathematical
  Sciences]}.
\newblock Springer-Verlag, Berlin, 1992.

\bibitem[PS00]{StanleyDeformations}
Alexander Postnikov and Richard~P. Stanley.
\newblock Deformations of {C}oxeter hyperplane arrangements.
\newblock {\em J. Combin. Theory Ser. A}, 91(1-2):544--597, 2000.
\newblock In memory of Gian-Carlo Rota.

\bibitem[Res16]{Mathematica}
Wolfram Research.
\newblock Mathematica, Version 10.4, 2016.

\bibitem[Sai75]{SaitoUniformization}
Kyoji Saito.
\newblock On the uniformization of complements of discriminant loci.
\newblock In {\em Conference Notes. Amer. Math. Soc. Summer Institute,
  Williamstown}, 1975.

\bibitem[Sai80]{SaitoLogDiff}
Kyoji Saito.
\newblock Theory of logarithmic differential forms and logarithmic vector
  fields.
\newblock {\em J. Fac. Sci. Univ. Tokyo Sect. IA Math.}, 27(2):265--291, 1980.

\bibitem[Sch97]{Spect}
Hal Schenck.
\newblock A spectral sequence for splines.
\newblock {\em Adv. in Appl. Math.}, 19(2):183--199, 1997.

\bibitem[Sch04]{HalLinearSystems}
Henry~K. Schenck.
\newblock Linear systems on a special rational surface.
\newblock {\em Math. Res. Lett.}, 11(5-6):697--713, 2004.

\bibitem[Sch14]{HalSplit}
Hal Schenck.
\newblock Splines on the {A}lfeld split of a simplex and type {A} root systems.
\newblock {\em J. Approx. Theory}, 182:1--6, 2014.

\bibitem[SS97]{LCoho}
Hal Schenck and Mike Stillman.
\newblock Local cohomology of bivariate splines.
\newblock {\em J. Pure Appl. Algebra}, 117/118:535--548, 1997.
\newblock Algorithms for algebra (Eindhoven, 1996).

\bibitem[ST98]{TeraoDoubleCoxeter}
Louis Solomon and Hiroaki Terao.
\newblock The double {C}oxeter arrangement.
\newblock {\em Comment. Math. Helv.}, 73(2):237--258, 1998.

\bibitem[Sta96]{StanleyIntervalOrders}
Richard~P. Stanley.
\newblock Hyperplane arrangements, interval orders, and trees.
\newblock {\em Proc. Nat. Acad. Sci. U.S.A.}, 93(6):2620--2625, 1996.

\bibitem[Ter02]{TeraoMultiDer}
Hiroaki Terao.
\newblock Multiderivations of {C}oxeter arrangements.
\newblock {\em Invent. Math.}, 148(3):659--674, 2002.

\bibitem[Wak07]{Wakamiko}
Atsushi Wakamiko.
\newblock On the exponents of 2-multiarrangements.
\newblock {\em Tokyo J. Math.}, 30(1):99--116, 2007.

\bibitem[Yos02]{YoshinagaPrimitiveDerivationMultiCoxeter}
Masahiko Yoshinaga.
\newblock The primitive derivation and freeness of multi-{C}oxeter
  arrangements.
\newblock {\em Proc. Japan Acad. Ser. A Math. Sci.}, 78(7):116--119, 2002.

\bibitem[Yos04]{YoshCharacterizationFreeArr}
Masahiko Yoshinaga.
\newblock Characterization of a free arrangement and conjecture of {E}delman
  and {R}einer.
\newblock {\em Invent. Math.}, 157(2):449--454, 2004.

\bibitem[Zie89]{ZieglerMulti}
G{\"u}nter~M. Ziegler.
\newblock Multiarrangements of hyperplanes and their freeness.
\newblock In {\em Singularities ({I}owa {C}ity, {IA}, 1986)}, volume~90 of {\em
  Contemp. Math.}, pages 345--359. Amer. Math. Soc., Providence, RI, 1989.

\end{thebibliography}
\bibliographystyle{alpha}

\appendix

\section{Two-Valued Families}\label{app}

In this appendix we illustrate pictorially the classification of Theorem~\ref{thm:intro} for two-valued multiplicities on $A_3$.  Given two positive integers $r$ and $s$, we assume $m_{ij}=r$ or $m_{ij}=s$ for all $i,j$.  In Table~\ref{tbl:TwoValuedIllustration}, the labeling of $K_4$ in the left column shows the assignment of multiplicities and the graph on the right shows which pairs $(r,s)$ correspond to free multiplicities (the obvious patterns continue).  The hollow dots represent free multiplicities, while the solid dots represent non-free multiplicities.  If present, the vertical line of free multiplicities along $r=1$ corresponds to multiplicities with a free vertex.  Free multiplicities clustered around the diagonal correspond to free ANN multiplicities.

	\begin{longtable}{cc}
		
		\begin{tikzpicture}[scale=1.4]
		\tikzstyle{dot}=[circle,fill=black,inner sep=1 pt];
		
		\node[dot] (0) at (0,0) {};
		\node[dot] (1) at (0,2){};
		\node[dot] (2) at (2,-1){};
		\node[dot] (3) at (-2,-1){};
		%\node at (-2/3,1/3){$\sigma_1$};
		%\node at (0,-2/3){$\sigma_2$};
		
		\draw (0)--node[right]{$s$}(1)-- node[left]{$r$}(3);
		\draw (0) --node[above]{$r$} (3) -- node[below]{$r$}(2)-- node[above]{$r$} (0);
		\draw (1)--node[right]{$r$} (2);
		\end{tikzpicture}
		&
		\includegraphics[width=.4\textwidth]{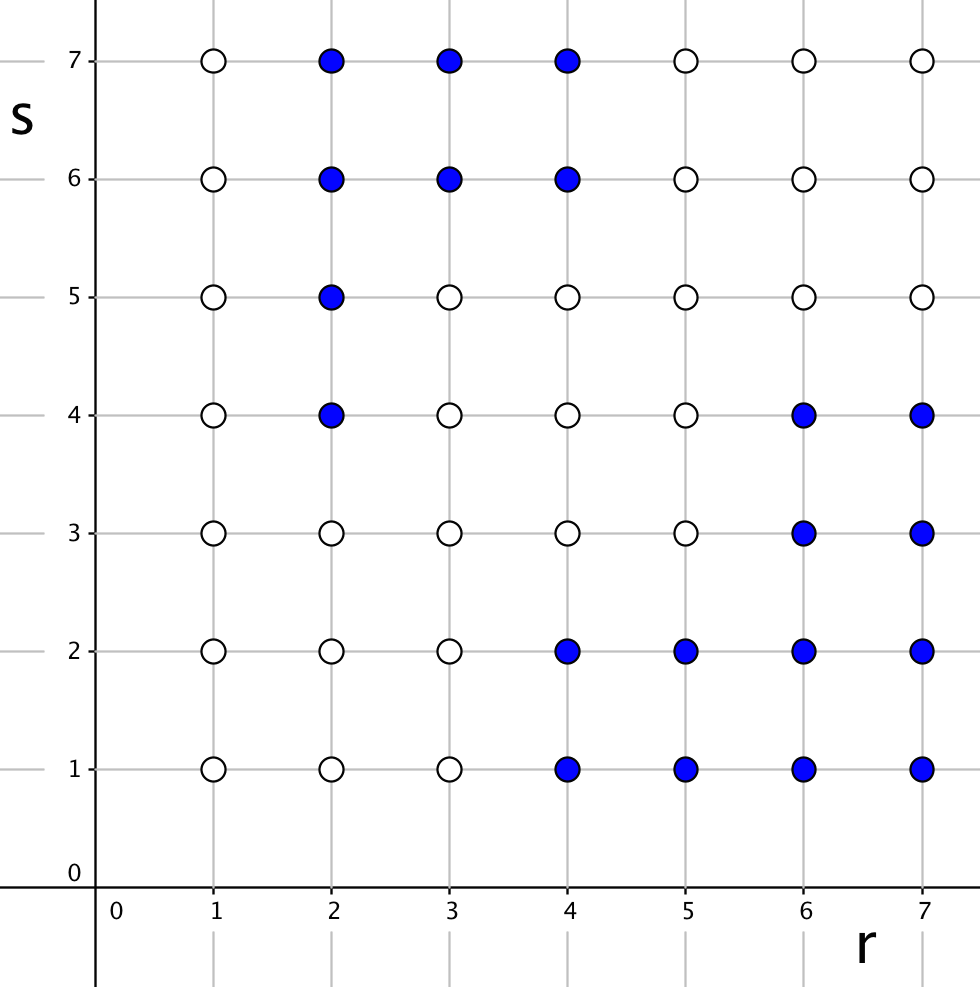}\\
		\begin{tikzpicture}[scale=1.4]
		\tikzstyle{dot}=[circle,fill=black,inner sep=1 pt];
		
		\node[dot] (0) at (0,0) {};
		\node[dot] (1) at (0,2){};
		\node[dot] (2) at (2,-1){};
		\node[dot] (3) at (-2,-1){};
		%\node at (-2/3,1/3){$\sigma_1$};
		%\node at (0,-2/3){$\sigma_2$};
		
		\draw (0)--node[right]{$s$}(1) -- node[left]{$r$}(3) ;
		\draw (0) --node[above]{$r$} (3) -- node[below]{$r$}(2)-- node[above]{$s$} (0);
		\draw (1)--node[right]{$r$} (2);
		\end{tikzpicture}
		&
		\includegraphics[width=.4\textwidth]{rands1s.png}\\
		
		\begin{tikzpicture}[scale=1.4]
		\tikzstyle{dot}=[circle,fill=black,inner sep=1 pt];
		
		\node[dot] (0) at (0,0) {};
		\node[dot] (1) at (0,2){};
		\node[dot] (2) at (2,-1){};
		\node[dot] (3) at (-2,-1){};
		%\node at (-2/3,1/3){$\sigma_1$};
		%\node at (0,-2/3){$\sigma_2$};
		
		\draw (0)--node[right]{$s$}(1)-- node[left]{$r$}(3);
		\draw (0) --node[above]{$s$} (3) -- node[below]{$r$}(2)-- node[above]{$s$} (0);
		\draw (1)--node[right]{$r$} (2);
		\end{tikzpicture}
		&
		\raisebox{50 pt}{\parbox[b]{.4\textwidth}{\centering Free for all $r,s\ge 1$}}\\
		
		\begin{tikzpicture}[scale=1.4]
		\tikzstyle{dot}=[circle,fill=black,inner sep=1 pt];
		
		\node[dot] (0) at (0,0) {};
		\node[dot] (1) at (0,2){};
		\node[dot] (2) at (2,-1){};
		\node[dot] (3) at (-2,-1){};
		%\node at (-2/3,1/3){$\sigma_1$};
		%\node at (0,-2/3){$\sigma_2$};
		
		\draw (0)--node[right]{$r$}(1) -- node[left]{$s$}(3);
		\draw (0) --node[above]{$s$} (3) -- node[below]{$r$}(2)-- node[above]{$r$} (0);
		\draw (1)--node[right]{$s$} (2);
		\end{tikzpicture}
		&
		\includegraphics[width=.4\textwidth]{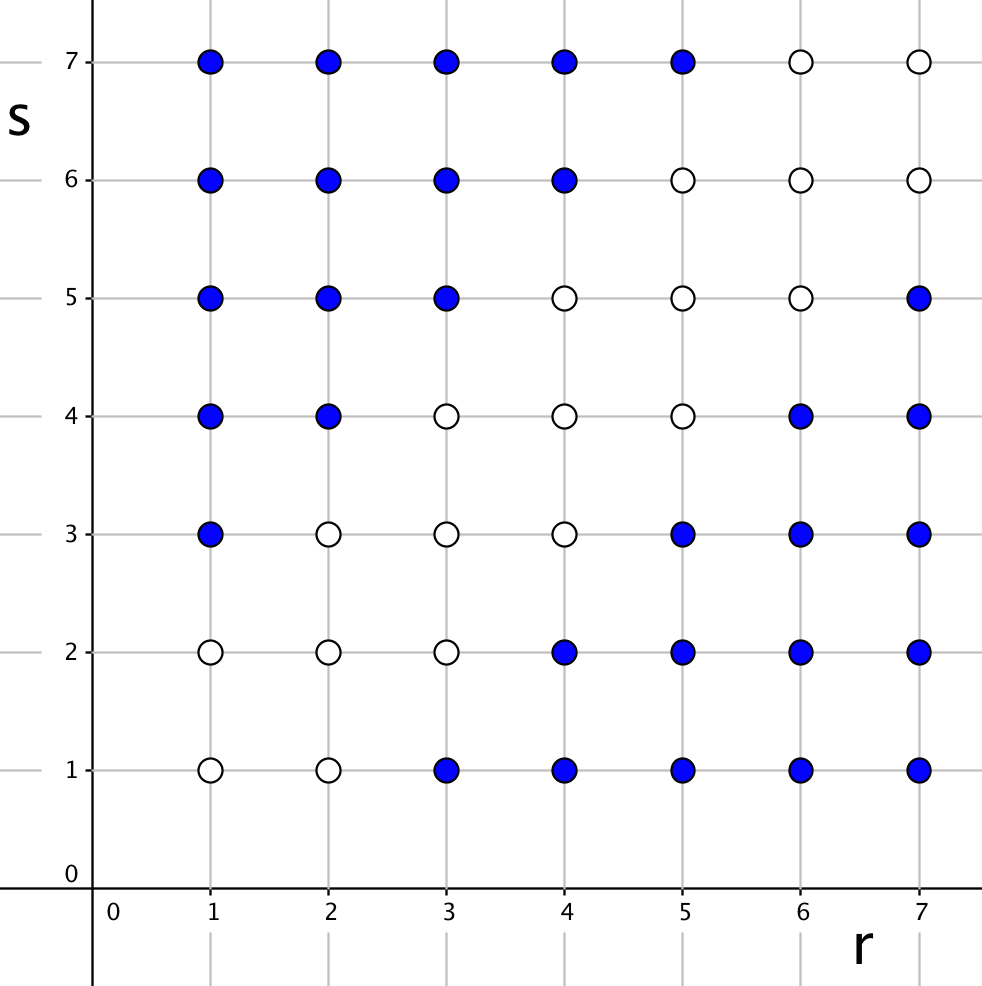}\\
		
		\begin{tikzpicture}[scale=1.4]
		\tikzstyle{dot}=[circle,fill=black,inner sep=1 pt];
		
		\node[dot] (0) at (0,0) {};
		\node[dot] (1) at (0,2){};
		\node[dot] (2) at (2,-1){};
		\node[dot] (3) at (-2,-1){};
		%\node at (-2/3,1/3){$\sigma_1$};
		%\node at (0,-2/3){$\sigma_2$};
		
		\draw (0)--node[right]{$r$}(1) -- node[left]{$r$}(3);
		\draw (0) --node[above]{$s$} (3) -- node[below]{$r$}(2)-- node[above]{$r$} (0);
		\draw (1)--node[right]{$s$} (2);
		\end{tikzpicture}
		&
		\includegraphics[width=.4\textwidth]{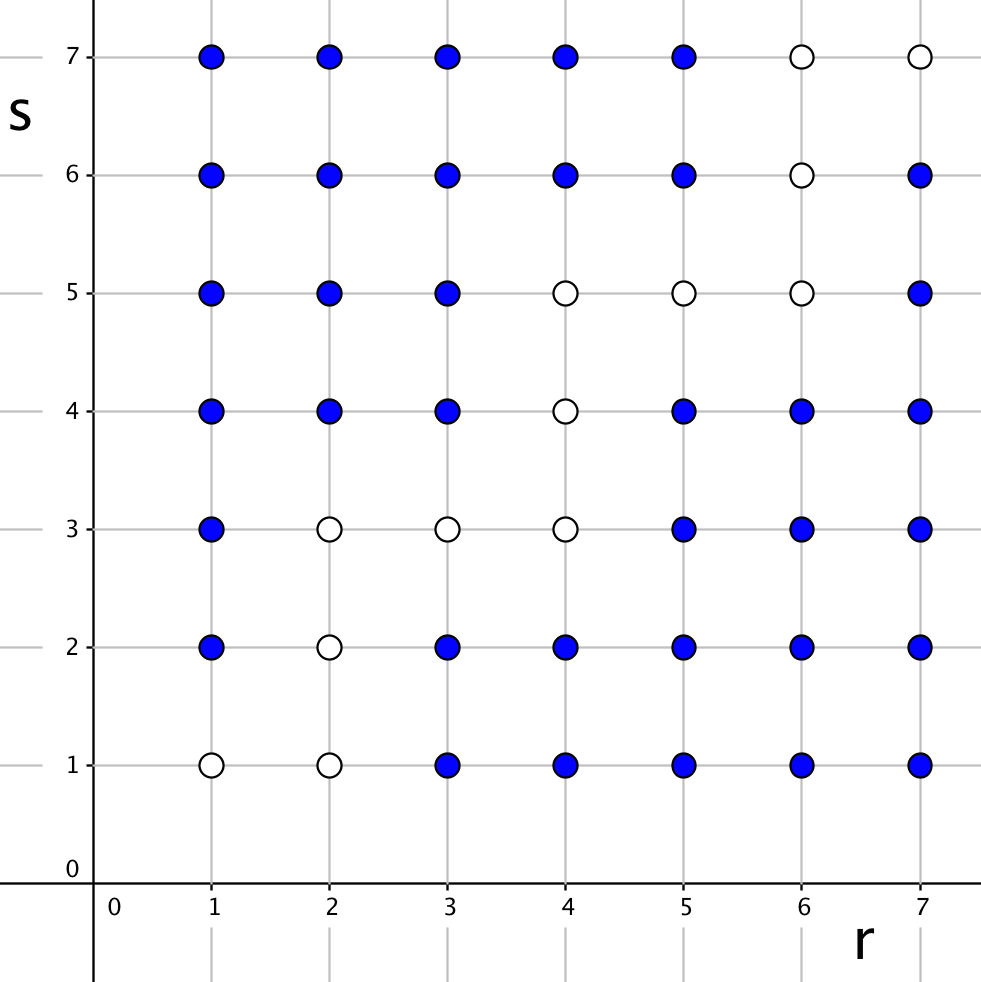}\\
		
		\caption{Free (hollow) and non-free (solid) two-valued multiplicities on $A_3$}\label{tbl:TwoValuedIllustration}
		
	\end{longtable}

\end{document}